\documentclass{article}%
\usepackage{amsfonts}
\usepackage{amsmath}
\usepackage{amssymb}
\usepackage{graphicx}%
\providecommand{\U}[1]{\protect\rule{.1in}{.1in}}
\newtheorem{theorem}{Theorem} [section]

\newtheorem{claim}[theorem]{Claim}

\newtheorem{conjecture}[theorem]{Conjecture}

\newtheorem{definition}[theorem]{Definition}
\newtheorem{example}[theorem]{Example}

\newtheorem{lemma}[theorem]{Lemma}

\newenvironment{proof}[1][Proof]{\noindent\textbf{#1.} }{\ \rule{0.5em}{0.5em}}
\begin{document}

\title{Hultman Numbers and Generalized Commuting Probability in Finite Groups}
\author{Yonah Cherniavsky, Vadim E. Levit, Robert Shwartz\\Department of Computer Science and Mathematics\\Ariel University, Israel\\\{yonah,levitv,robertsh\}@ariel.ac.il
\and Avraham Goldstein\\Department of Mathematics\\City University of New York, NY, USA\\avraham.goldstein.nyc@gmail.com}
\date{}
\maketitle

\begin{abstract}
Let $G$ be a finite group and $\pi$ be a permutation from $S_{n}$. We
investigate the distribution of the probabilities of the equality
\[
a_{1}a_{2}\cdots a_{n-1}a_{n}=a_{\pi_{1}}a_{\pi_{2}}\cdots a_{\pi_{n-1}}%
a_{\pi_{n}}%
\]
when $\pi$ varies over all the permutations in $S_{n}$. The probability
\[
Pr_{\pi}(G)=Pr(a_{1}a_{2}\cdots a_{n-1}a_{n}=a_{\pi_{1}}a_{\pi_{2}}\cdots
a_{\pi_{n-1}}a_{\pi_{n}})
\]
is identical to $Pr_{1}^{\omega}(G)$, with
\[
\omega=a_{1}a_{2}...a_{n-1}a_{n}a_{\pi_{1}}^{-1}a_{\pi_{2}}^{-1}\cdots
a_{\pi_{n-1}}^{-1}a_{\pi_{n}}^{-1},
\]
as it is defined in \cite{DasNath1} and \cite{NathDash1}. The notion of
commutativity degree, or the probability of a permutation equality $a_{1}%
a_{2}=a_{2}a_{1}$, for which $n=2$ and $\pi=\langle2\;\;1\rangle$, was
introduced and assessed by P. Erd\"{o}s and P. Turan in \cite{ET} in 1968 and
by W. H. Gustafson in \cite{G} in 1973. In \cite{G} Gustafson establishes a
relation between the probability of $a_{1},a_{2}\in G$ commuting and the
number of conjugacy classes in $G$. In this work we define several other
parameters, which depend only on a certain interplay between the conjugacy
classes of $G$, and compute the probabilities of general permutation
equalities in terms of these parameters. It turns out that this probability,
for a permutation $\pi$, depends only on the number $c(Gr(\pi))$ of the
alternating cycles in the cycle graph $Gr(\pi)$ of $\pi$. The cycle graph of a
permutation was introduced by V. Bafna and P. A. Pevzner in \cite{BP}. We
describe the spectrum of the probabilities of permutation equalities in a
finite group as $\pi$ varies over all the elements of $S_{n}$. This spectrum
turns-out to be closely related to the partition of $n!$ into a sum of the
corresponding Hultman numbers.

\medskip


\end{abstract}

\section{Introduction}

Study of the probability that two random elements in a finite group $G$
commute is very natural \cite{ET}, \cite{G}, \cite{GerRob}. In 1968 Erd\"{o}s
and Turan proved that%
\[
Pr(a_{1}a_{2}=a_{2}a_{1})>\frac{log(log|G|))}{|G|}.
\]
In early 1970s Dixon observed that the commuting probability is $\leq\frac
{1}{12}$ for every finite non-Abelian simple group (this was submitted as a
problem in Canadian Mathematical Bulletin \textbf{13} (1970), with a solution
appearing in 1973). In 1973 Gustafson proved that the commuting probability is
equal to $\frac{k(G)}{|G|}$, where $k(G)$ is the number of conjugacy classes
in $G$ \cite{G}. Based on this observation, further Gustafson obtained the
upper bound of the commuting probability in any finite non-Abelian group to be
$\frac{5}{8}$ \cite{G}. This upper bound is actually attained in many finite
groups, including $D_{4}$ and $Q_{8}$.

Significant amount of work has been done in assessing the commuting
probability for various special cases of finite groups. For example, the
commuting probability for dihedral groups is studied in \cite{L}, for direct
product of dihedral groups is studied in \cite{CGK}, for wreath products of
two Abelian groups is studied in \cite{ErSu}, for non-solvable groups is
studied in \cite{GerRob}. More information on the development of this subject
and its applications may be found in \cite{Dixon}.

There has been done a lot of various probabilistic studies for finite groups.
Many of these studies can be regarded as various generalizations of the
commuting probability problem. For example, establishing the number of ordered
$k$-tuples of elements of group $G$ which have pairwise commuting elements
\cite{ES}. Another example of generalization is finding the probability that
the commutator of two random group elements of $G$ is equal to a given element
\cite{PS} or that two random elements of $G$ are conjugate \cite{BBW}.

In their recent works \cite{NathDash1} and \cite{DasNath1} Das and Nath study
the probability $Pr_{g}^{\omega}(G)$ of the equality
\[
a_{1}a_{2}...a_{n-1}a_{n}a_{\pi_{1}}^{-1}a_{\pi_{2}}^{-1}\cdots a_{\pi_{n-1}%
}^{-1}a_{\pi_{n}}^{-1}=g
\]
in a finite group $G$. The word
\[
a_{1}a_{2}...a_{n-1}a_{n}a_{\pi_{1}}^{-1}a_{\pi_{2}}^{-1}\cdots a_{\pi_{n-1}%
}^{-1}a_{\pi_{n}}^{-1},
\]
in which $a_{1}a_{2}...a_{n-1}a_{n}$ vary over all the elements of $G$, is
denoted by $\omega$. Thus this is a generalization of the classical study of
the commuting probability, for which case $\omega=a_{1}a_{2}a_{2}^{-1}%
a_{1}^{-1}$ and $g=1$.

In this paper we take a sligthly different approach in generalizing the study
of the commuting probability. Let
\[
\pi=\langle\pi_{1}\;\;\pi_{2}\;\;...\;\;\pi_{n}\rangle
\]
be a permutation from $S_{n}$, written in a shortened way of the two row
notation. We define $Pr_{\pi}(G)$ as the probability of the equality
\[
a_{1}a_{2}\cdots a_{n}=a_{\pi_{1}}a_{\pi_{2}}\cdots a_{\pi_{n-1}}a_{\pi_{n}}%
\]
in $G$. Notice that $Pr_{\langle2\;\;1\rangle}(G)$ is just the commuting
probability of $G$.

Notice also, that the probability%

\[
Pr_{\pi}(G)=Pr(a_{1}a_{2}\cdots a_{n-1}a_{n}=a_{\pi_{1}}a_{\pi_{2}}\cdots
a_{\pi_{n-1}}a_{\pi_{n}})
\]
is identical to $Pr_{1}^{\omega}(G)$, with
\[
\omega=a_{1}a_{2}...a_{n-1}a_{n}a_{\pi_{1}}^{-1}a_{\pi_{2}}^{-1}\cdots
a_{\pi_{n-1}}^{-1}a_{\pi_{n}}^{-1},
\]
as it is defined in \cite{DasNath1} and \cite{NathDash1}.

\begin{itemize}
\item We obtain a new description of $Pr_{\pi}(G)$ in terms of non-negative
integers\newline$c_{i_{1},...,i_{n};j}(G)$, which count the number of ways in
which an element from a conjugacy class $\Omega_{j}$ of $G$ can be broken into
a product of elements from the conjugacy classes $\Omega_{1_{1}}%
,...,\Omega_{i_{n}}$ of $G$.

\item We show that the probability $Pr$, for a fixed group $G$, depends only
on the number of alternating cycles in the cycle graph of the permutation
$\pi$.

\item Finally, we obtain that spectrum of the probabilities of permutation
equalities in a finite group, as $\pi$ varies over all the elements of $S_{n}%
$, is closely related to the partition of $n!$ into a sum of the corresponding
Hultman numbers.
\end{itemize}

\subsection{Definitions and Notations}

Let $n$ be a natural number, $S_{n}$ be the group of all permutations of $n$
elements and
\[
\pi=\langle\pi_{1}\;\;\pi_{2}\;\;...\;\;\pi_{n}\rangle
\]
be a permutation from $S_{n}$, written in a shortened way of the two row
notation. Sometimes we also use the cyclic notation for elements of $S_{n}$,
for which we use parentheses and commas. Thus, for example, $(\theta
_{1},\theta_{2},\theta_{3})$ represents the cycle $\theta_{1}\mapsto\theta
_{2}\mapsto\theta_{3}\mapsto\theta_{1}$. Let $G$ be a finite group.

For any set $S$, we denote the size of $S$ by $\left\vert S\right\vert $. The
number of conjugacy classes of $G$ is denoted by $c(G)$. We denote them by
$\Omega_{1},...,\Omega_{c(G)}$. For $g\in G$ we denote the conjugacy class of
$g$, i.e. the set of all the elements of the form $hgh^{-1}$, where $h\in G$,
by $\Omega(g)$ and the centralizer of $g$, i.e. the set of all elements of $G$
which commute with $g$, by $C_{G}(g)$. Recall, that
\[
\left\vert G\right\vert =\left\vert \Omega(g)\right\vert \cdot\left\vert
C_{G}(g)\right\vert
\]
for all $g$. The set $\{\Omega_{1},...,\Omega_{c(G)}\}$ is denoted by $C(G)$.
We denote by $G^{\prime}$the commutator subgroup of $G$. Namely, $G^{\prime}$
is the minimal normal subgroup of $G$ which contains all the elements of the
form $ghg^{-1}h^{-1}$, where $g,h\in G$. For $g,h\in G$, we denote by $g(h)$
the element $ghg^{-1}$. To indicate that $g,h\in G$ are conjugate we write
$g\sim h$.

By $D_{4}$ we denote the dihedral group with $8$ elements. By $Q_{8}$ we
denote the multiplicative group of unit quaternions, which also has $8$
elements in it.

\begin{definition}
For a sequence $(g_{1},g_{2},\dots,g_{n})$ of $n$ elements of $G$ we denote
by\newline$Stab.Prod_{n}(g_{1},g_{2},\dots,g_{n})$ the set of all the
sequences $(a_{1},a_{2},\dots,a_{n})$ of $n$ elements of $G$ such that
\[
a_{1}^{-1}g_{1}a_{1}\cdot a_{2}^{-1}g_{2}a_{2}\cdots a_{n}^{-1}g_{n}%
a_{n}=g_{1}\cdot g_{2}\cdots g_{n}.
\]

\end{definition}

Notice that $Stab.Prod_{n}(g_{1},g_{2},...,g_{n})$ is a generalization of the
centralizer of an element and $Stab.Prod_{1}(g)$ is just $C_{G}(g)$.

\begin{definition}
The nonnegative integer $c_{i_{1},...,i_{n};j}(G)$ is the number of different
ways of breaking any fixed element $y\in\Omega_{j}(G)$ into a product
$y=x_{1}x_{2}\cdots x_{n}$, so that each $x_{t}$, where $1\leq t\leq n$,
belongs to the class $\Omega_{i_{t}}(G)$.
\end{definition}

Notice, that $c_{i_{1},...,i_{n};j}(G)$ does not depend on the choice of the
element $y\in\Omega_{j}(G)$. Indeed, if we take some other $y^{\prime}%
\in\Omega_{j}(G)$ then exists some $g\in G$ such that $y^{\prime}=gyg^{-1}$
and $y=g^{-1}y^{\prime}g$. Then each product $y=x_{1}x_{2}\cdots x_{n}$
corresponds to the product
\[
y^{\prime}=gyg^{-1}=(gx_{1}g^{-1})(gx_{2}g^{-1})\cdots(gx_{n}g^{-1}),
\]
in which each $x_{t}^{\prime}=(gx_{t}g^{-1})$ belongs to the same class
$\Omega_{i_{t}}(G)$ as $x_{t}$. Vice verse, each product $y^{\prime}%
=x_{1}^{\prime}x_{2}^{\prime}\cdots x_{n}^{\prime}$ corresponds back to the
product
\[
y=g^{-1}y^{\prime}g=(g^{-1}x_{1}^{\prime}g)(g^{-1}x_{2}^{\prime}%
g)\cdots(g^{-1}x_{n}^{\prime}g).
\]
Thus, we see that the number of such different products is the same for $y$
and $y\prime$ and that it depends only on the equivalence class $\Omega
_{j}(G)$. Notice, that $c_{i_{1},...,i_{n};j}(G)$ can be zero and that
$c_{j;j}(G)=1$.

\begin{definition}
We denote by $L_{\pi}(G)$ the number of different solutions of the
equation\newline%
\[
a_{1}a_{2}\cdots a_{n-1}a_{n}=a_{\pi_{1}}a_{\pi_{2}}\cdots a_{\pi_{n-1}}%
a_{\pi_{n}}%
\]
in $G$. We denote by by $Pr_{\pi}(G)$ the probability
\[
Pr(a_{1}a_{2}\cdots a_{n-1}a_{n}=a_{\pi_{1}}a_{\pi_{2}}\cdots a_{\pi_{n-1}%
}a_{\pi_{n}})
\]
that the equation
\[
a_{1}a_{2}\cdots a_{n-1}a_{n}=a_{\pi_{1}}a_{\pi_{2}}\cdots a_{\pi_{n-1}}%
a_{\pi_{n}}%
\]
is satisfied by the random elements $a_{1},a_{2},...,a_{n-1},a_{n}$ of $G$.
\end{definition}

Clearly, $Pr_{\pi}(G)=\frac{L_{\pi}(G)}{|G|^{n}}$.

\begin{definition}
We define $Pr^{n}(G)$ as $Pr_{\langle n\;\;n-1\;\;...\;\;2\;\;1\rangle}(G)$.
\end{definition}

Let $\pi$ be a permutation from $S_{n}$ and let
\[
\omega=a_{1}\cdots a_{n}\cdot a_{\pi_{1}}^{-1}\cdots a_{\pi_{n}}^{-1}.
\]
Then $Pr_{\pi}(G)$ is identical to $Pr_{1}^{\omega}(G)$, as defined in
\cite{DasNath1}, \cite{DasNath2}, \cite{NathDash1}. Similarly, our $Pr^{n}(G)$
is identical to $Pr_{1}^{n}(G)$ in these works. For further information on
calculations, properties and estimates of $Pr_{1}^{\omega}(G)$ and $Pr_{1}%
^{n}(G)$ we refer te reader to \cite{DasNath1}, \cite{DasNath2},
\cite{NathDash1} as well.

\begin{definition}
We denote by $Spec_{n}(G)$ the set of all $Pr_{\pi}(G)$, as $\pi$ runs over
all the permutations from $S_{n}$.
\end{definition}

For the information on the Hultman numbers and the related definitions and
notations we refer to \cite{DL}. Here we briefly review these notions.

\begin{definition}
\label{def.cycle.graph} The cycle graph $Gr(\phi)$ of a permutation $\phi\in
S_{n}$ is the bi-colored directed graph
\end{definition}

with $n+1$ vertices $\phi_{0}=0,\phi_{1},...,\phi_{n}$, whose edge set
consists of:

\begin{itemize}
\item black edges $\phi_{n}\rightarrow\phi_{n-1}\rightarrow...\rightarrow
\phi_{0}\rightarrow\phi_{n}$, and

\item grey edges $0\dashrightarrow1\dashrightarrow2\dashrightarrow
...\dashrightarrow n\dashrightarrow0$.
\end{itemize}

The set of black and grey edges decomposes in a unique way into edge-disjoint
alternating cycles - the cycles in $Gr(\phi)$, which alternate black and grey edges.

\begin{definition}
\label{def.hult.num} The Hultman number $S_{H}(n,k)$ counts the number of
permutations in $S_{n}$ whose cycle graph decomposes into $k$ alternating cycles.
\end{definition}

\begin{definition}
Denote by $\phi^{\cdot}$ the big cycle in $S(1+n)=Sym(\{0,1,2,...,n\})$,
composed of the black arrows of $Gr(\phi)$.
\end{definition}

We will use the cyclic notation $(\phi_{0},\phi_{n},\phi_{n-1},...,\phi
_{2},\phi_{1})$ for the big black cycle $\phi^{\cdot}$. Notice, that there is
a trivial one-to-one correspondence between the permutations of $S_{n}$ and
big cycles in $S(1+n)$. Namely, the entries of a big cycle, starting from the
one after $0$, are interpreted as the entries of the permutation, written in
the shortened way of the two row notation. Thus, for any big cycle in
$S(1+n)$, we can easily obtain the unique permutation in $S_{n}$, for which
this big cycle is its big black cycle.

\begin{definition}
Let $\phi^{\circ}\in S(1+n)$ be $\phi^{\cdot}\cdot(0,1,\dots,n)$.
\end{definition}

Let $H(S_{n})$ be the Hultman decomposition of $S_{n}$ into pairwise disjoint
sets, each containing all the permutations with the same number of alternating
cycles in their cycle graph.

\begin{theorem}
\cite{DL} \label{Theorem 8}The cycle decomposition of $\phi^{\circ}$ contains
the same number of cycles as the number of alternating cycles of $Gr(\phi)$.
\end{theorem}

Notice, that for any permutation $\phi$ and any $m=0,1,...,n$, we cannot have
$\phi^{\circ}(m)=m+1$, since otherwise we would get%
\[
\phi^{\cdot}(x+1)=\left[  \phi^{\circ}\cdot(n,n-1,\dots,0)\right]  (x+1)=x+1,
\]

which is contradiction to $\phi^{\circ}$ being an $(n+1)$-cycle.

Let $\phi\in S_{n}$ be a permutation.

\begin{definition}
\label{def.exchange} For any four numbers $0\leq x,y,w,z\leq n$, such that
$z\rightarrow x\dashrightarrow x+1\rightarrow y$ and $y\rightarrow w$ are
present in some cycles of $Gr(\phi)$, the $x--y$ exchange operation is defined
as follows:

\begin{itemize}
\item if $x=w$ or $y=z$ then $x--y$ exchange operation does not do anything to
$\phi$;

\item if $y=\phi_{i}$ and $x=\phi_{j}$ and $x+1=\phi_{i+1}$, where $i+1<j$,
then $x--y$ exchange operation changes\newline$\langle\phi_{1}\;\;...\;\;\phi
_{i-1}\;\;\phi_{i}\;\;\phi_{i+1}\;\;...\phi_{j-1}\;\;\phi_{j}\;\;...\phi
_{n}\rangle$ to $\langle\phi_{1}\;\;...\;\;\phi_{i-1}\;\;\phi_{i+1}%
\;\;...\phi_{j}\;\;\phi_{i}\;\;\phi_{j+1}\;\;...\phi_{n}\rangle$;

\item if $y=\phi_{i}$ and $x=\phi_{j}$ and $x+1=\phi_{i+1}$, where $j<i$, then
$x--y$ exchange operation changes\newline$\langle\phi_{1}\;\;...\;\;\phi
_{j}\;\;\phi_{j+1}\;\;...\phi_{i-1}\;\;\phi_{i}\;\;\phi_{i+1}\;\;...\phi
_{n}\rangle$ to $\langle\phi_{1}\;\;...\;\;\phi_{j}\;\;\phi_{i}\;\;\phi
_{j+1}\;\;...\phi_{i-1}\;\;\phi_{i+1}\;\;...\phi_{n}\rangle$;

\item if $y=0$ and $x=\phi_{j}$ and $x+1=\phi_{1}$, then $x--y$ exchange
operation changes\newline$\langle\phi_{1}\;\;...\;\;\phi_{j-1}\;\;\phi
_{j}\;\;\phi_{j+1}\;\;...\phi_{n}\rangle$ to $\langle\phi_{j+1}\;\;\phi
_{j+2}\;\;...\phi_{n}\;\;\phi_{1}\;\;...\phi_{j-1}\;\;\phi_{j}\rangle$.
\end{itemize}
\end{definition}

\begin{example}
For instance,

\begin{itemize}
\item If $\phi=\langle4\;\;1\;\;6\;\;2\;\;5\;\;7\;\;3\rangle$, and we apply
$5--1$ exchange operation, we get the permutation $\phi_{\lbrack5--1]}%
=\langle4\;\;6\;\;2\;\;5\;\;1\;\;7\;\;3\rangle$.

\item If $\phi=\langle4\;\;1\;\;6\;\;2\;\;5\;\;7\;\;3\rangle$ and we apply
$4--2$ exchange operation, we get the permutation $\phi_{\lbrack4--2]}%
=\langle4\;\;2\;\;1\;\;6\;\;5\;\;7\;\;3\rangle$.

\item If $\phi=\langle4\;\;6\;\;1\;\;2\;\;5\;\;7\;\;3\rangle$ and we apply
$0--6$ exchange operation, we get the permutation $\phi_{\lbrack0--6]}%
=\langle6\;\;4\;\;1\;\;2\;\;5\;\;7\;\;3\rangle$.

\item If $\phi=\langle4\;\;1\;\;6\;\;3\;\;5\;\;7\;\;2\rangle$ and we apply
$3--0$ exchange operation, we get the permutation $\phi_{\lbrack3--0]}%
=\langle5\;\;7\;\;2\;\;4\;\;1\;\;6\;\;3\rangle$. \newline
\end{itemize}
\end{example}

Notice, that in the alternating cycles of the cycle graph of the new
permutation, obtained after performing the $x--y$ exchange operation, we will
have $y\rightarrow x\dashrightarrow x+1\rightarrow w$ and $z\rightarrow y$.

\begin{definition}
Two permutation $\phi,\theta\in S_{n}$ are in the same \textquotedblleft$x--y$
exchange orbit" if exist permutations $\tau_{1},...,\tau_{k}$ such that
$\tau_{1}=\phi$, $\tau_{k}=\theta$, and, for each $i=1,...,k-1$, either
$\tau_{i}$ can be obtained from $\tau_{i+1}$ by an $x--y$ exchange operation
or $\tau_{i+1}$ can be obtained from $\tau_{i}$ by an $x--y$ exchange operation.
\end{definition}

Notice, that according to our definition, $\phi,\theta\in S_{n}$ can be in the
same $x--y$ exchange orbit, while neither one of them can be obtained from the
other one by several $x--y$ exchange operations.

\begin{definition}
\label{cyclic-cdot}For $\phi\in S_{n}$, such that the big black cycle
$\phi^{\cdot}\in S(1+n)$ contains $(x+1)\rightarrow y\rightarrow x$, which, if
$x=n$, becomes $0\rightarrow y\rightarrow n$, we define $x--y$ cyclic
operation as follows:

\begin{itemize}
\item If $y>x+1$, then in $\phi^{\cdot}$ we replace $x+1$ with $y-1$ and each
$t$, where $t=x+2,...,y-1$, we replace with $t-1$;

\item If $y<x$, then in $\phi^{\cdot}$ we replace $x$ with $y+1$ and each $t$,
for $t=y+1,...,x-1$, we replace with $t+1$;
\end{itemize}
\end{definition}

\begin{example}
For instance,

\begin{itemize}
\item If $\phi=\langle6\;\;5\;\;3\;\;1\;\;4\;\;2\rangle$, then $\phi^{\cdot
}=(0,2,4,1,3,5,6)$. We can perform $1--4$ cyclic operation on $\phi$ and
obtain a new permutation $\theta$. We have $\theta^{\cdot}=(0,3,4,1,2,5,6)$.
Hence, $\theta=\langle6\;\;5\;\;2\;\;1\;\;4\;\;3\rangle$.

\item If $\phi=\langle4\;\;1\;\;5\;\;2\;\;6\;\;3\rangle$, then $\phi^{\cdot
}=(0,3,6,2,5,1,4)$. We can perform $5--2$ cyclic operation on $\phi$ and
obtain a new permutation $\theta$. We have $\theta^{\cdot}=(0,4,6,2,3,1,5)$.
Hence, $\theta=\langle5\;\;1\;\;3\;\;2\;\;6\;\;4\rangle$.

\item If $\phi=\langle4\;\;1\;\;5\;\;2\;\;6\;\;3\rangle$, then $\phi^{\cdot
}=(0,3,6,2,5,1,4)$. We can perform $0--4$ cyclic operation on $\phi$ and
obtain a new permutation $\theta$. We have $\theta^{\cdot}=(0,2,6,1,5,3,4)$.
Hence, $\theta=\langle4\;\;3\;\;5\;\;1\;\;6\;\;2\rangle$.

\item If $\phi=\langle4\;\;1\;\;5\;\;2\;\;6\;\;3\rangle$, then $\phi^{\cdot
}=(0,3,6,2,5,1,4)$. We can perform $6--3$ cyclic operation on $\phi$ and
obtain a new permutation $\theta$. We have $\theta^{\cdot}=(0,3,4,2,6,1,5)$.
Hence, $\theta=\langle4\;\;3\;\;5\;\;1\;\;6\;\;2\rangle$.

\item If $\phi=\langle4\;\;1\;\;5\;\;2\;\;6\;\;3\rangle$, then $\phi^{\cdot
}=(0,3,6,2,5,1,4)$. We can perform $3--0$ cyclic operation on $\phi$ and
obtain a new permutation $\theta$. We have $\theta^{\cdot}=(0,1,6,3,5,2,4)$.
Hence, $\theta=\langle4\;\;2\;\;5\;\;3\;\;6\;\;1\rangle$.
\end{itemize}
\end{example}

\begin{definition}
Two permutation $\phi,\theta\in S_{n}$ are in the same \textquotedblleft$x--y$
cyclic orbit" if exist permutations $\tau_{1},...,\tau_{k}$ such that
$\tau_{1}=\phi$, $\tau_{k}=\theta$, and, for each $i=1,...,k-1$, either
$\tau_{i}$ can be obtained from $\tau_{i+1}$ by an $x--y$ cyclic operation or
$\tau_{i+1}$ can be obtained from $\tau_{i}$ by an $x--y$ cyclic operation.
\end{definition}

Notice, that according to our definition, $\phi,\theta\in S_{n}$ can be in the
same $x--y$ exchange orbit, while neither one of them can be obtained from the
other one by several $x--y$ cyclic operations.\newline

Notice, that belonging to \textquotedblleft the same orbit" extends both
$x--y$ exchange and $x--y$ cyclic operations to equivalence relations.

\begin{definition}
Two permutation $\phi,\theta\in S_{n}$ are called \textquotedblleft$x--y$
equivalent" if exist permutations $\tau_{1},...,\tau_{k}$ such that $\tau
_{1}=\phi$, $\tau_{k}=\theta$, and, for each $i=1,...,k-1$, $\tau_{i}$ and
$\tau_{i+1}$ are in the same $x--y$ exchange orbit or in the same $x--y$
cyclic orbit.
\end{definition}

\begin{lemma}
\label{howcyclicworks} Let $\phi\in S_{n}$ be such that $\phi^{\cdot}\in
S(1+n)$ is of the form $(...,x+1,y,x,...)$. Let $\theta\in S_{n}$ be obtained
from $\phi$ by an $x--y$ cyclic operation. Then $\phi^{\circ}$ is transformed
by that $x--y$ cyclic operation to $\theta^{\circ}$ in the following way:
First write $\phi^{\circ}\in S(1+n)$ in the cyclic notation. Then:

\begin{enumerate}
\item If $y>x+1$:

\begin{itemize}
\item Instead of $x+1$ write the pair $y-1,x$ (replace $x$ in a cycle, which
contains it, by $y-1\mapsto x$);

\item For all $t=x+2,...,y-2$, instead of $t$ write $t-1$;

\item Instead of the pair $y-1,x$ write $y-2$ (replace $y-1\mapsto x$ in a
cycle, which contains it, by $y$).
\end{itemize}

\item If $y<x$:

\begin{itemize}
\item Instead of the pair $x,y$ write $y+1$ (replace $x\mapsto y$ in a cycle,
which contains it, by $y+1$);

\item For all $t=y+2,...,x-2$, instead of $t$ write $t+1$;

\item Instead of $x-1$ write the pair $x,y$ (replace $x-1$ in a cycle, which
contains it, by $x\mapsto y$).
\end{itemize}
\end{enumerate}
\end{lemma}

\begin{proof}
Definition \ref{cyclic-cdot} implies that the permutation $\phi^{\circ}%
=\phi^{\cdot}\cdot(0,1,...,n)$ has a cycle, which contains
\[
x\mapsto y=x\dashrightarrow(x+1)\rightarrow y,
\]
and a cycle, which contains
\[
y-1\mapsto x=(y-1)\dashrightarrow y\rightarrow x.
\]
Thus, $\phi^{\circ}$ has a cycle, which contains $y-1\mapsto x\mapsto y$.

If $y>x+1$ then, to obtain $\theta^{\cdot}$, we replaced $x+1$ in $\phi
^{\cdot}$ with $y-1$, and each $t$ in $\phi^{\cdot}$, for $t=x+2,...,y-1$, we
replaced with $t-1$. Let us consider just the following two replacements in
$\phi^{\cdot}$ $-$ of $x+1$ with $y-1$ and of $x+2$ with $x+1$ (without, for
now, performing all the other replacements). Assume, that before these two
replacements, for some elements $u,v$ of $\phi^{\cdot}$, we had $u\rightarrow
(x+1)$ and $(x+2)\rightarrow v$ in $\phi^{\cdot}$. Then $\phi^{\circ}$, before
performing these two replacements, had a cycle, which contained
\[
(u-1)\mapsto(x+1)\mapsto v.
\]
Replacing in that cycle $x+1$ with $y-1$ and $x+2$ with $x+1$ will change
\[
(u-1)\mapsto(x+1)\mapsto v
\]
to
\[
(u-1)\mapsto(y-1)\mapsto x\mapsto v.
\]
A careful consideration of how all the other replacements affect the cycles of
$\phi^{\circ}$, inside which they are performed, yields our Lemma for the case
$y>x+1$.

If $y<x$ then, to obtain $\theta^{\cdot}$, we replaced we replaced $x$ in
$\phi^{\cdot}$ with $y+1$, and each $t$, for $t=y+1,...,x-1$, we replaced with
$t+1$. Let's consider just the following two replacements in $\phi^{\cdot}$
$-$ of $x$ with $y+1$ and of $y+1$ with $y+2$ (without, for now, performing
all the other replacements). Assume, that before these two replacements, for
some element $v$ of $\phi^{\cdot}$, we had $(y+1)\rightarrow v$ in
$\phi^{\cdot}$. Then $\phi^{\circ}$, before performing these two replacements,
had a cycle, which contained
\[
(y-1)\mapsto x\mapsto y\mapsto v.
\]
Replacing in that cycle $x$ with $y+1$ and $y+1$ with $y+1$ will change
\[
y-1\mapsto x\mapsto y\mapsto v
\]
in that cycle to
\[
y-1\mapsto y+1\mapsto v.
\]
A careful consideration of how all the other replacements affect the cycles of
$\phi^{\circ}$, inside which they are performed, yields our Lemma for the case
$y<x$.
\end{proof}

\begin{definition}
\cite{Hall1940}\label{def.isoclinic} Two groups $G_{1}$ and $G_{2}$ are called
isoclinic, if the following three conditions holds

\begin{itemize}
\item There exists an isomorphism $\alpha$ from $G_{1}/Z(G_{1})$ onto
$G_{2}/Z(G_{2})$;

\item There exists an isomorphism $\beta$ from the commutator subgroup
$G_{1}^{\prime}$ to the commutator subgroup $G_{2}^{\prime}$;

\item If $\alpha(a_{1}Z_{1})=a_{2}Z_{2}$, and $\alpha(b_{1}Z_{1})=b_{2}Z_{2}$,
then necessarily
\[
\beta(a_{1}^{-1}b_{1}^{-1}a_{1}b_{1})=a_{2}^{-1}b_{2}^{-1}a_{2}b_{2}.
\]

\end{itemize}
\end{definition}

For instance, every two abelian groups are isoclinic. The Dihedral group
$D_{4}$ and the Quaternion group $Q_{8}$ give an example of isoclinism for
non-abelian groups.

\begin{definition}
\cite{Buckley2014} Two groups $G_{1}$ and $G_{2}$ are called weakly isoclinic,
if the first two conditions of Definition \ref{def.isoclinic} holds true, but
the third condition of Definition \ref{def.isoclinic} does not necessarily hold.
\end{definition}

\section{Preliminaries}

The following well-known results, which we reproduce in the Lemmas
\ref{commute-conjug} and \ref{conjug-decomp}, are crucial for our work.

\begin{lemma}
\label{commute-conjug} For any $a,b\in G$ we have $ab\sim ba$.
\end{lemma}

\begin{proof}
$b(ab)b^{-1}=babb^{-1}=ba.$
\end{proof}

\begin{lemma}
\label{conjug-decomp} For any $x$ and $y$ from the same conjugacy class
$\Omega_{i}$ of $G$ there are
\[
\frac{|G|}{|\Omega_{i}|}=\left\vert C_{G}(x)\right\vert =\left\vert
C_{G}(y)\right\vert
\]
different ways to break $x$ into a product $x=ab$ of elements $a,b\in G$ so
that $ba=y$.
\end{lemma}

\begin{proof}
Since $x\sim y$ there exists some $b\in G$ such that $bxb^{-1}=y$. If we set
$a=xb^{-1}$, we get $ab=xb^{-1}b=x$ and $ba=bxb^{-1}=y$. For each pair
$a\prime,b\prime$ of elements of $G$, such that $a^{\prime}b^{\prime}=x$,
there exists a unique element $g=b^{\prime}b^{-1}=\left(  a^{\prime}\right)
^{-1}a$ of $G$, such that $a^{\prime}=ag^{-1}$ and $b^{\prime}=gb$. Now
\[
b\prime a\prime=gbag^{-1}=gyg^{-1}.
\]
Hence, the pairs $a^{\prime},b^{\prime}$ of elements of $G$ such that
$a^{\prime}b^{\prime}=x$ and $b^{\prime}a^{\prime}=y$, are in one-to-one
correspondence with the elements $g$ from $C_{G}(Y)$. So, the number of pairs
$a^{\prime},b^{\prime}$ of elements of $G$ ,such that $a^{\prime}b^{\prime}=x$
and $b^{\prime}a^{\prime}=y$, is equal to $|C_{G}(y)|$. But
\[
\left\vert C_{G}(y)\right\vert =\frac{|G|}{|\Omega_{i}|}=|C_{G}(x)|.
\]

\end{proof}

The classical result on the commute probability (see \cite{G}) follows immediately.

\begin{theorem}
\label{commute} $Pr^{2}(G)=Pr(a_{1}a_{2}=a_{2}a_{1})=\frac{c(G)}{|G|}$.
\end{theorem}

\begin{proof}
For each $x\in G$ there are exactly $\frac{|G|}{|\Omega(x)|}$ different ways
to write $ab=x=ba$, where $a,b\in G$. Thus, for each $\Omega_{i}$ there are
$|G|$ different ways to write $ab=x=ba$, where $a,b\in G$ and $x\in\Omega(i)$.
Indeed, there are $|\Omega(x)|$ different elements in $\Omega_{i}$ and for
each one of them there are $\frac{|G|}{|\Omega(x)|}$ different ways to break
them into a product of commuting elements. Thus, $L_{\langle2\;\;1\rangle
}(G)=\left\vert G\right\vert \cdot c(G)$ and
\[
Pr^{2}(G)=\frac{\left\vert G\right\vert \cdot c(G)}{|G|^{2}}=\frac{c(G)}%
{|G|}.
\]

\end{proof}

\section{Calculation of $Spec_{4}(G)$}

Before addressing the general case in the following sections, we study here
the spectrum of probabilities for permutations from $S_{4}$. The material in
this section is self-contained and will help to illustrate the general case.
We begin with the following Lemma, which is a particular case of the general
fact, observed by Das and Nath \cite{DasNath2}, that if
\[
\omega_{1}=a_{1}a_{2}...a_{2n}a_{1}^{-1}a_{2}^{-1}\cdots a_{2n}^{-1}%
\]
and
\[
\omega_{2}=a_{1}a_{2}...a_{2n+1}a_{1}^{-1}a_{2}^{-1}\cdots a_{2n+1}^{-1}%
\]
then $Pr_{g}^{\omega_{1}}(G)=Pr_{g}^{\omega_{2}}(G)$.

\begin{lemma}
\label{threecommute} $Pr^{3}(G)=Pr_{\langle3\;\;2\;\;1\rangle}%
(G)=Pr(abc=cba)=Pr^{2}(G)=\frac{|c(G)|}{|G|}$.
\end{lemma}

\begin{proof}
From $abc=cba$, we get $abcb=cbab$. Thus, the commutator $[cb,ab]=1$.
Therefore, if we choose and fix an element $b$, then the two elements $a$ and
$c$ should be chosen in a such way that $[ab,cb]=1$. Hence, for every of an
element $a$, we must choose an element $c$ in such a way, that $cb\in
C_{G}(ab)$. Since $G$ is a group, the number of different elements of the for
$ab\in G$, for a fixed $b$, is $|G|$. Similarly, the number of different
elements of the form $cb$ in each $C_{G}(ab)$ is $|C_{G}(ab)|$. Since $b$ runs
through all elements of $G$, we get
\[
L_{\langle3\;\;2\;\;1\rangle}(G)=\left\vert G\right\vert \cdot\sum
\limits_{a\in G}\left\vert C_{G}(a)\right\vert =|G|^{2}\cdot|C(G)|.
\]
This implies
\[
Pr_{\langle3\;\;2\;\;1\rangle}(G)=\frac{|C(G)|}{|G|}=Pr_{\langle2\;\;1\rangle
}(G).
\]

\end{proof}

Notice, that for all the five non-trivial permutations from $S_{3}$ we get the
same probability for the corresponding permutational equality. Also observe,
that these five permutations have two alternating cycles in their cycle graph,
while the identity permutation has four alternating cycles in its cycle graph.

\begin{theorem}
\label{specthree} $Spec_{3}(G)=\{\frac{c(G)}{|G|},1\}$, where $Pr(a_{1}%
a_{2}a_{3}=a_{\pi_{1}}a_{\pi_{2}}a_{\pi_{3}})=\frac{c(G)}{|G|}$ for every
$\pi\neq Id$ in $S_{3}$.
\end{theorem}

\begin{proof}
By Lemma \ref{threecommute}, $Pr(a_{1}a_{2}a_{3}=a_{3}a_{2}a_{1})=\frac
{c(G)}{|G|}$. Clearly,
\[
Pr(a_{1}a_{2}a_{3}=a_{2}a_{1}a_{3})=Pr(a_{1}a_{2}=a_{2}a_{1})=\frac{c(G)}{|G|}%
\]
and
\[
Pr(a_{1}a_{2}a_{3}=a_{1}a_{3}a_{2})=Pr(a_{2}a_{3}=a_{3}a_{2})=\frac{c(G)}%
{|G|}.
\]

To compute $Pr(a_{1}a_{2}a_{3}=a_{3}a_{1}a_{2})$ notice that if we denote the
product $a_{1}a_{2}$ by $g$ then, as $a_{1}$ and $a_{2}$ run through all
elements of $G$, their product $g$ will become equal to every element of $G$
exactly $|G|$ times. Thus, the number $L(a_{1}a_{2}a_{3}=a_{3}a_{1}a_{2})$ of
different solutions of the equation $a_{1}a_{2}a_{3}=a_{3}a_{1}a_{2}$ in $G$
is $|G|$ times the number of different solutions of equation $ga_{3}=a_{3}g$.
Thus
\[
Pr(a_{1}a_{2}a_{3}=a_{3}a_{1}a_{2})=\frac{\left\vert G\right\vert
\cdot\left\vert G\right\vert \cdot c(G)}{|G|^{3}}=\frac{c(G)}{|G|}.
\]

The same argument, but with denoting $g=a_{2}a_{3}$, shows that
\[
Pr(a_{1}a_{2}a_{3}=a_{2}a_{3}a_{1})=\frac{c(G)}{|G|}.
\]
Obviously, $Pr(a_{1}a_{2}a_{3}=a_{1}a_{2}a_{3})=1$.
\end{proof}

\begin{lemma}
For any non-identity permutation $\pi\in S_{4}$, if $\pi_{1}=1$ or $\pi_{4}%
=4$, then $Pr_{\pi}(G)=\frac{c(G)}{|G|}$.
\end{lemma}

\begin{proof}
If $\pi_{1}=1$ then the equation $a_{1}a_{2}a_{3}a_{4}=a_{\pi_{1}}a_{\pi_{2}%
}a_{\pi_{3}}a_{\pi_{4}}$ is equivalent to the equation $a_{2}a_{3}a_{4}%
=a_{\pi_{2}}a_{\pi_{3}}a_{\pi_{4}}$. And, by Theorem \ref{specthree}, there
are $|G|^{2}\cdot c(G)$ different equations $a_{2}a_{3}a_{4}=a_{\pi_{2}}%
a_{\pi_{3}}a_{\pi_{4}}$ in $G$. To each one of these equations corresponds
exactly $|G|$ different equations of the form $a_{1}a_{2}a_{3}a_{4}=a_{\pi
_{1}}a_{\pi_{2}}a_{\pi_{3}}a_{\pi_{4}}$ (since $a_{1}$ can be any element of
$G$). Thus $L_{\pi}(G)=|G|^{3}\cdot c(G)$ and $Pr_{\pi}(G)=\frac{c(G)}{|G|}$.
The same argument, but with $a_{4}$ instead of $a_{1}$, is applied for
$\pi_{4}=4$.
\end{proof}

\begin{lemma}
For any non-identity permutation $\pi\in S_{4}$, if $\pi_{i+1}=\pi_{1}+1$
(where $4+1$ becomes $1$), for some $1\leq i\leq3$, then $Pr_{\pi}%
(G)=\frac{c(G)}{|G|}$.
\end{lemma}

\begin{proof}
The condition $\pi_{i+1}=\pi_{i}+1$ means that $\pi$ keeps two consequent
numbers $\pi_{i}$ and $\pi_{i}+1$ in their consequent order. Such are, for
example, the permutations $\langle4\;\;3\;\;1\;\;2\rangle,\langle
3\;\;4\;\;1\;\;2\rangle,$\newline$\langle2\;\;3\;\;4\;\;1\rangle,$
$\langle4\;\;2\;\;3\;\;1\rangle,\langle4\;\;1\;\;2\;\;3\rangle$. Denote the
product $a_{\pi_{i}}a_{\pi_{i+1}}$ by $g$. Then, as $a_{\pi_{i}}$ and
$a_{\pi_{i+1}}$ run over all elements of $G$, the product $g$ will become
equal to each element of $G$ exactly $|G|$ times. By Theorem
\ref{threecommute} the new equation (with $g$ instead of $a_{\pi_{i}}%
a_{\pi_{i+1}}$) has exactly $|G|^{2}\cdot c(G)$ solutions. But each solution
is obtained $G$ times, as $a_{\pi_{i}}$ and $a_{\pi_{i+1}}$ run over all
elements of $G$. Thus $L_{\pi}(G)=|G|^{3}\cdot c(G)$ and $Pr_{\pi}%
(G)=\frac{c(G)}{|G|}$.
\end{proof}

We showed that for fifteen permutations $\pi\in S_{4}$, the probability of the
corresponding equality satisfies $Pr_{\pi}(G)=\frac{c(G)}{|G|}$. Notice, that
these fifteen permutations have exactly three alternating cycles in their
cycle graph $Gr(\pi)$. Clearly, for the identity permutation, which has five
alternating cycles in its cycle graph, the probability of the corresponding
equality is $1$. Now we will show that the remaining eight permutations in
$S_{4}$, which have one alternating cycle in their cycle graph, their
corresponding equalities all have the same probability, which, in general, is
different from $1$ and from $\frac{c(G)}{|G|}$.

\begin{theorem}
\label{mainfour}
\begin{gather*}
Pr(a_{1}a_{2}a_{3}a_{4}=a_{2}a_{1}a_{4}a_{3})=Pr(a_{1}a_{2}a_{3}a_{4}%
=a_{4}a_{3}a_{2}a_{1})=Pr^{4}(G)=\\
=\frac{\sum\limits_{x,y\in G}|Stab.Prod_{2}(x,y)|}{|G|^{4}}=\sum
\limits_{i,k,j=1}^{c(G)}\frac{\left\vert \Omega_{j}\right\vert \cdot
c_{i,k;j}^{2}(G)}{\left\vert \Omega_{i}\right\vert \cdot\left\vert \Omega
_{k}\right\vert \cdot\left\vert G\right\vert ^{2}}=\sum\limits_{i,k,j=1}%
^{c(G)}\frac{\left\vert \Omega_{j}\right\vert \cdot c_{i,k;j}(G)\cdot
c_{k,i;j}(G)}{\left\vert \Omega_{i}\right\vert \cdot\left\vert \Omega
_{k}\right\vert \cdot\left\vert G\right\vert ^{2}}.
\end{gather*}

\end{theorem}

\begin{proof}
Notice, that for any $x,y\in G$, the set of all ordered pairs $(g,h)$ of
elements of $G$, such that $xy=g^{-1}xgh^{-1}yh$, which was denoted by
$Stab.Prod_{2}(x,y)$, has an alternative description as the set of all the
ordered pairs $(g\in G,f^{-1}\in G)$, such that $gxyf=xgfy$.

Thus, for each ordered pair $(x,y)$ of elements of $G$, $|Stab.Prod_{2}(x,y)|$
is the number of different equations $gxyf=xgfy$, where $g,h\in G$.

As $x$ and $y$ run over all elements of $G$, we obtain that the total number
of different equations $gxyf=xgfy$ in $G$ is $\sum\limits_{x,y\in
G}|Stab.Prod_{2}(x,y)|$. Consequently,%
\[
Pr(a_{1}a_{2}a_{3}a_{4}=a_{2}a_{1}a_{4}a_{3})=\frac{\sum\limits_{x,y\in
G}|Stab.Prod_{2}(x,y)|}{|G|^{4}}.
\]

Next, let us consider a generic equation $a_{1}a_{2}a_{3}a_{4}=a_{4}a_{3}%
a_{2}a_{1}$. Denote the product $a_{1}a_{2}$ by $x$, the product $a_{2}a_{1}$
by $x^{\prime}$, the product $a_{3}a_{4}$ by $y$ and the product $a_{4}a_{3}$
by $y^{\prime}$. The equation $a_{1}a_{2}a_{3}a_{4}=a_{4}a_{3}a_{2}a_{1}$
becomes $xy=y^{\prime}x^{\prime}$. Notice, that by Lemma \ref{commute-conjug},
we have $x\sim x^{\prime}$ and $y\sim y^{\prime}$. Now consider any
$x,x^{\prime},y,y^{\prime}\in G$, such that $x\sim x\prime$ and $y\sim
y^{\prime}$ and that $xy=y^{\prime}x^{\prime}$. Then, by Lemma
\ref{conjug-decomp}, there are $\frac{|G|}{|\Omega_{(}x)|}=|C_{G}(x)|$
different ways of breaking $x$ into a product $a_{1}a_{2}$ in such a way, that
$x^{\prime}=a_{2}a_{1}$. Also, by Lemma \ref{conjug-decomp}, there are
$\frac{|G|}{|\Omega_{(}y)|}=|C_{G}(y)|$ different ways of breaking $y$ into a
product $a_{3}a_{4}$ in such a way, that $y^{\prime}=a_{4}a_{3}$. Thus, to
each fixed equation $xy=y^{\prime}x^{\prime}$ correspond $\left\vert
C_{G}(x)\right\vert \cdot\left\vert C_{G}(y)\right\vert $ different equations
$a_{1}a_{2}a_{3}a_{4}=a_{4}a_{3}a_{2}a_{1}$.

Notice, that for any fixed $x,y\in G$, we can take $x^{\prime\prime}%
=(y^{-1}xy)$ and obtain an equation $xy=y(y^{-1}xy)=yx^{\prime\prime}$. Now,
for any equation $xy=y^{\prime}x^{\prime}$ as above, there exist some $g,h\in
G$ such that $y^{\prime}=hyh^{-1}$ and $x^{\prime}=gx^{\prime\prime}g^{-1}$.
Hence, if we select and fix $x$ and $y$, we will have
\[
\frac{|Stab.Prod_{2}(y,x^{\prime\prime})|}{\left\vert C_{G}(y)\right\vert
\cdot\left\vert C_{G}(x^{\prime\prime})\right\vert }%
\]
different equations $xy=y^{\prime}x^{\prime}$, in which $x^{\prime}\sim x$ and
$y^{\prime}\sim y$, and any two such equations are equal if and only if $h\in
C_{G}(y)$ and $g\in C_{G}(x^{\prime\prime})$. Notice, that $\left\vert
C_{G}(x^{\prime\prime})\right\vert =|C_{G}(x)|$. Thus, for each fixed ordered
pair $(x,y)$ of elements of $G$ we have
\[
\frac{|Stab.Prod_{2}(y,(y^{-1}xy))|}{\left\vert C_{G}(y)\right\vert
\cdot\left\vert C_{G}(x)\right\vert }%
\]
different equations $xy=y^{\prime}x^{\prime}$. To each one of these equations
correspond $\left\vert C_{G}(x)\right\vert \cdot\left\vert C_{G}(y)\right\vert
$ different equations $a_{1}a_{2}a_{3}a_{4}=a_{4}a_{3}a_{2}a_{1}$, as shown
above. Thus, to each ordered pair $x$ and $y$ of elements of $G$ correspond
$|Stab.Prod_{2}(y,(y^{-1}xy))|$ different equations $a_{1}a_{2}a_{3}%
a_{4}=a_{4}a_{3}a_{2}a_{1}$.

Thus to find $Pr(a_{1}a_{2}a_{3}a_{4}=a_{4}a_{3}a_{2}a_{1})$ we need to sum
$|Stab.Prod_{2}(y,(y^{-1}xy))|$ over all $x,y\in G$. Since $x$ and $y$ run
through all the elements of $G$, this is the same as to sum over all $u,v\in
G$ the value $|Stab.Prod_{2}(u,v)|$. Thus, we obtain
\[
L_{\langle4\;\;3\;\;2\;\;1\rangle}(G)=\sum\limits_{x,y\in G}\left\vert
Stab.Prod_{2}(x,y)\right\vert
\]
and
\begin{gather*}
Pr^{4}(G)=Pr(a_{1}a_{2}a_{3}a_{4}=a_{4}a_{3}a_{2}a_{1})=\\
=\frac{\sum\limits_{x,y\in G}|Stab.Prod_{2}(x,y)|}{|G|^{4}}=Pr(a_{1}a_{2}%
a_{3}a_{4}=a_{2}a_{1}a_{4}a_{3}).
\end{gather*}

Now, select and fix an element $z$ in some equivalence class $\Omega_{j}$ of
$G$. For each (fixed) equation $xy=z=x^{\prime}y^{\prime}$ we have exactly
$\left\vert C_{G}(x)\right\vert \cdot|C_{G}(y)|$ different equations
\[
(a_{1}a_{2})(a_{3}a_{4})=xy=z=x^{\prime}y^{\prime}=(a_{2}a_{1})(a_{4}a_{3}).
\]

There are $|C_{G}(x)|$ different ways to break $x$ into a product $a_{1}a_{2}$
in such a way that $a_{2}a_{1}=x^{\prime}$. Similarly, There are $|C_{G}(y)|$
different ways to break $y$ into a product $a_{3}a_{4}$ in such a way that
$a_{4}a_{3}=y^{\prime}$. It is satisfied:
\[
\left\vert C_{G}(x)\right\vert \cdot\left\vert C_{G}(y)\right\vert
=\frac{|G|^{2}}{\left\vert \Omega(x)\right\vert \cdot\left\vert \Omega
(y)\right\vert }.
\]
There are $c_{i,k;j}(G)$ different ways to break $z$ as $xy$ and
$c_{i,k;j}(G)$ different ways to break $z$ as $x^{\prime}y^{\prime}$. Thus,
for each pair $\Omega_{i},\Omega_{k}$ of conjugacy classes of $G$, there are
$c_{i,k;j}^{2}(G)$ different equations $xy=z=x^{\prime}y^{\prime}$ with
$x,x^{\prime}\in\Omega_{i}$ and $y,y^{\prime}\in\Omega_{k}$. Thus we obtain%
\[
Pr(a_{1}a_{2}a_{3}a_{4}=a_{2}a_{1}a_{4}a_{3})=\sum\limits_{i,k,j=1}%
^{c(G)}\frac{\left\vert \Omega_{j}\right\vert \cdot c_{i,k;j}^{2}%
(G)}{\left\vert \Omega_{i}\right\vert \cdot\left\vert \Omega_{k}\right\vert
\cdot\left\vert G\right\vert ^{2}}.
\]

Finally, select and fix an element $z$ in some equivalence class $\Omega_{j}$
of $G$. For each (fixed) equation $xy=z=y^{\prime}x^{\prime}$ we have exactly
$\left\vert C_{G}(x)\right\vert \cdot\left\vert C_{G}(y)\right\vert $
different equations
\[
(a_{1}a_{2})(a_{3}a_{4})=xy=z=y^{\prime}x^{\prime}=(a_{4}a_{3})(a_{2}a_{1}).
\]
and
\[
\left\vert C_{G}(x)\right\vert \cdot\left\vert C_{G}(y)\right\vert
=\frac{|G|^{2}}{\left\vert \Omega(x)\right\vert \cdot\left\vert \Omega
(y)\right\vert }.
\]

Now, for each pair $\Omega_{i},\Omega_{k}$ of conjugacy classes of $G$, there
are $c_{i,k;j}(G)\cdot c_{k,i;j}(G)$ different equations $xy=z=y^{\prime
}x^{\prime}$ with $x,x^{\prime}\in\Omega_{i}$ and $y,y^{\prime}\in\Omega_{k}$.
Thus we obtain%
\[
Pr(a_{1}a_{2}a_{3}a_{4}=a_{4}a_{3}a_{2}a_{1})=\sum\limits_{i,k,j=1}%
^{c(G)}\frac{\left\vert \Omega_{j}\right\vert \cdot c_{i,k;j}(G)\cdot
c_{k,i;j}(G)}{\left\vert \Omega_{i}\right\vert \cdot\left\vert \Omega
_{k}\right\vert \cdot\left\vert G\right\vert ^{2}}.
\]

\end{proof}

\begin{theorem}
The following are valid:

\begin{itemize}
\item $Pr(a_{1}a_{2}a_{3}a_{4}=a_{3}a_{2}a_{4}a_{1})=Pr^{4}(G)$;

\item $Pr(a_{1}a_{2}a_{3}a_{4}=a_{4}a_{2}a_{1}a_{3})=Pr^{4}(G)$;

\item $Pr(a_{1}a_{2}a_{3}a_{4}=a_{2}a_{4}a_{3}a_{1})=Pr^{4}(G)$;

\item $Pr(a_{1}a_{2}a_{3}a_{4}=a_{4}a_{1}a_{3}a_{2})=Pr^{4}(G)$.
\end{itemize}
\end{theorem}

\begin{proof}%
\begin{gather*}
Pr^{4}(G)=Pr(a_{1}a_{2}a_{3}a_{4}=a_{4}a_{3}a_{2}a_{1})=\\
Pr(a_{4}{}^{-1}a_{1}a_{2}a_{3}a_{4}a_{1}^{-1}=a_{2}{}^{-1}a_{2}a_{3}%
a_{2})=\newline Pr(t^{-1}s^{-1}xts=u^{-1}xu),
\end{gather*}
where $t=a_{4}$, $s=a_{1}{}^{-1}$, $x=a_{2}a_{3}$, and $u=a_{2}$. Now,
\begin{gather*}
Pr(a_{1}a_{2}a_{3}a_{4}=a_{3}a_{2}a_{4}a_{1})=Pr(a_{2}a_{3}a_{4}=a_{1}{}%
^{-1}a_{3}a_{2}a_{4}a_{1})=\\
Pr(a_{4}{}^{-1}a_{2}a_{3}a_{4}=a_{4}{}^{-1}a_{1}{}^{-1}a_{3}a_{2}a_{4}%
a_{1})=\\
Pr\left(  a_{4}{}^{-1}a_{3}{}^{-1}a_{3}a_{2}a_{3}a_{4}=a_{4}{}^{-1}a_{1}%
{}^{-1}a_{3}a_{2}a_{4}a_{1}\right)  =Pr(\left(  t^{\prime}\right)
^{-1}\left(  s^{\prime}\right)  ^{-1}x^{\prime}t^{\prime}s^{\prime}=\left(
u^{\prime}\right)  ^{-1}x^{\prime}u^{\prime}),
\end{gather*}
where $t^{\prime}=a_{4}^{\prime}$, $s^{\prime}=a_{1}$, $x^{\prime}=a_{3}a_{2}%
$, $u^{\prime}=a_{3}a_{4}$. Again, as $a_{2}$, $a_{3}$ and $a_{4}$ run through
all elements of $G$, $t^{\prime}$, $x^{\prime}$ and $u^{\prime}$ also run
through all elements of $G$ and there is a one-to-one correspondence between
choosing $a_{1}$, $a_{2}$ and $a_{3}$ and choosing $t^{\prime}$, $x^{\prime}$
and $u^{\prime}$. Thus, we see that
\[
Pr(a_{1}a_{2}a_{3}a_{4}=a_{3}a_{2}a_{4}a_{1})=Pr^{4}(G).
\]
Notice, that $Pr(a_{1}a_{2}a_{3}a_{4}=a_{3}a_{2}a_{4}a_{1})$ corresponds to
the permutation $\langle3\;\;2\;\;4\;\;1\rangle$, which, in the cyclic
notation, is $(1,3,4)(2)$. The other three permutations, listed in this
Theorem, are $(1,4,3)(2)$, $(1,2,4)(3)$, and $(1,4,2)(3)$. The result for
$(1,2,4)(3)$ follows by symmetry. The result for $(1,4,3)(2)$ is easily proved
in a way, similar to our proof for $(1,3,4)(2)$. Namely,
\begin{gather*}
Pr(a_{1}a_{2}a_{3}a_{4}=a_{4}a_{2}a_{1}a_{3})=Pr(a_{4}^{-1}a_{1}a_{2}%
a_{3}a_{4}=a_{2}a_{1})=\\
Pr(a_{3}{}^{-1}a_{4}{}^{-1}a_{1}a_{2}a_{3}a_{4}=a_{3}{}^{-1}a_{2}a_{1}%
a_{3})=\\
Pr\left(  a_{3}{}^{-1}a_{4}{}^{-1}a_{1}a_{2}a_{3}a_{4}=a_{3}{}^{-1}a_{1}%
{}^{-1}a_{1}a_{2}a_{1}a_{3}\right)  =Pr(\left(  t^{\prime}\right)
^{-1}\left(  s^{\prime}\right)  ^{-1}x^{\prime}t^{\prime}s^{\prime}=\left(
u^{\prime}\right)  ^{-1}x^{\prime}u^{\prime}),
\end{gather*}
where $t^{\prime}=a_{3}^{\prime}$, $s^{\prime}=a_{4}$, $x^{\prime}=a_{1}a_{2}%
$, $u^{\prime}=a_{1}a_{3}$. The result for $(1,4,2)(3)$ now follows by symmetry.
\end{proof}

\begin{theorem}%
\[
Pr(a_{1}a_{2}a_{3}a_{4}=a_{2}a_{4}a_{1}a_{3})=Pr^{4}(G)
\]

and
\[
Pr(a_{1}a_{2}a_{3}a_{4}=a_{3}a_{1}a_{4}a_{2})=Pr^{4}(G).
\]

\end{theorem}

\begin{proof}
First, we claim that%
\[
Pr(a_{1}a_{2}a_{3}a_{4}=a_{2}a_{4}a_{1}a_{3})=Pr(a_{2}^{-1}a_{1}a_{2}%
a_{3}a_{4}a_{3}^{-1}=a_{4}a_{1})=Pr(t_{1}t_{2}t_{3}t_{4}=t_{4}t_{3}t_{2}%
t_{1}),
\]
where $t_{1}=a_{2}^{-1}a_{1}$, $t_{2}=a_{2}$, $t_{3}=a_{3}a_{4}$ and
$t_{4}=a_{3}^{-1}$. Clearly, as $a_{1}$ and $a_{2}$ run through all elements
of $G$, $t_{1}$ and $t_{2}$ also run through all elements of $G$ and there is
a one-to-one correspondence between choosing $a_{1}$ and $a_{2}$ and choosing
$t_{1}$ and $t_{2}$. Similarly, as $a_{3}$ and $a_{4}$ run through all
elements of $G$, $t_{3}$ and $t_{4}$ also run through all elements of $G$ and
there is a one-to-one correspondence between choosing $a_{3}$ and $a_{4}$ and
choosing $t_{3}$ and $t_{4}$.
\[
Pr(a_{1}a_{2}a_{3}a_{4}=a_{3}a_{1}a_{4}a_{2})=Pr(a_{2}a_{3}=a_{1}^{-1}%
a_{3}a_{1}a_{4}a_{2}a_{4}^{-1})=Pr(t_{1}t_{2}t_{3}t_{4}=t_{4}t_{3}t_{2}%
t_{1}),
\]
where $t_{1}=a_{4}^{-1}$, $t_{2}=a_{4}a_{2}$, $t_{3}=a_{1}$ and $t_{4}%
=a^{-1}a_{3}$. Again, as $a_{1}$ and $a_{3}$ run through all elements of $G$,
$t_{1}$ and $t_{3}$ also run through all elements of $G$ and there is a
one-to-one correspondence between choosing $a_{1}$ and $a_{3}$ and choosing
$t_{1}$ and $t_{3}$. Similarly, as $a_{2}$ and $a_{4}$ run through all
elements of $G$, $t_{2}$ and $t_{4}$ also run through all elements of $G$ and
there is a one-to-one correspondence between choosing $a_{2}$ and $a_{4}$ and
choosing $t_{2}$ and $t_{4}$.
\end{proof}

Now we compute $Pr(a_{1}a_{2}a_{3}a_{4}=a_{2}a_{1}a_{4}a_{3})=Pr^{4}(G)$ for
the cases $G=D_{4}$ and $G=Q_{8}$. These computations are made only to
\textquotedblleft justify" our Lemma \ref{SPEC4} below. These results, but in
a much more general form, and for a much larger variety of groups, were
obtained in \cite{DasNath1} and \cite{NathDash1}. No matter $G=D_{4}$ or
$G=Q_{8}$ we get that:

\begin{itemize}
\item $Pr^{4}(G)=\frac{17}{32}$, which is, actually, smaller than
$Pr(a_{1}a_{2}=a_{2}a_{1})=\frac{5}{8}$;

\item the center of the group consists of the identity $1$ and another element
$c$, such that $c^{2}=1$;

\item the factor of the group by its center is Abelian.
\end{itemize}

Consequently, if $xy\neq yx$ then $xy=cyx$. Now, $a_{1}a_{2}a_{3}a_{4}%
=a_{2}a_{1}a_{4}a_{3}$ either if $a_{1}a_{2}=a_{2}a_{1}$ and $a_{3}a_{4}%
=a_{4}a_{3}$, or if $a_{1}a_{2}=ca_{2}a_{1}$ and $a_{3}a_{4}=ca_{4}a_{3}$. We
conclude with the claim that%

\[
Pr(a_{1}a_{2}a_{3}a_{4}=a_{2}a_{1}a_{4}a_{3})=Pr^{4}(G)=\frac{5}{8}\cdot
\frac{5}{8}+\frac{3}{8}\cdot\frac{3}{8}=\frac{17}{32}%
\]
for for either $D_{4}$ or $Q_{8}$.

\begin{lemma}
\label{SPEC4} $Spec_{4}(G)=\{1,\frac{|c(G)|}{|G|}=Pr^{2}(G),Pr_{\langle
4\;\;3\;\;2\;\;1\;\;\rangle}(G)=Pr^{4}(G)\}$
\end{lemma}

\begin{proof}
The lemmas and theorems of this section established, that for any permutation
from $S_{4}$, the corresponding probability is either $1$ or $\frac
{|c(G)|}{|G|}$ or $Pr^{4}(G)$.

The calculations for $G=D_{4}$ and $G=Q_{8}$ show, that these three numbers
are pairwise different for these two groups.
\end{proof}

Notice, that for any two permutations from $S_{2}$, $S_{3}$ or $S_{4}$, their
permutational equalities have the same probability if and only if these
permutations have the same number of alternating cycles in their cycle graphs.
This fact was established above for permutations from $S_{3}$ and $S_{4}$, and
is trivially verified for permutations from $S_{2}$. Consequently, the number
of different permutations, corresponding to each probability in $Spec_{2}(G)$,
$Spec_{3}(G)$ or $Spec_{4}(G)$, is exactly the Hultman number $S_{H}(n,k)$,
where $n=2,3,4$ and $k$ is the number of alternating cycles in the cycle
graphs of these permutations. Now we address the general case.

\section{Probabilities of permutation equalities, number of alternating cycles
and Hultman decomposition}

Again, we refer to \cite{DL} for all the relevant information on cycle graphs,
the Hultman decomposition, and Hultman numbers.

\begin{theorem}
\label{exchange} Let $\phi$ and $\theta$ be two permutations in $S_{n}$, such
that $\theta$ is obtained from $\phi$ by an $x--y$ exchange operation. Then
\[
Pr(a_{1}a_{2}\cdots a_{n}=a_{\phi_{1}}a_{\phi_{2}}\cdots a_{\phi_{n}%
})=Pr(a_{1}a_{2}\cdots a_{n}=a_{\theta_{1}}a_{\theta_{2}}\cdots a_{\theta_{n}%
}).
\]

\end{theorem}

\begin{proof}
If $x=y$ then $\theta=\phi$ and the Theorem follows. Hence, we assume that
$x\neq y$.

First, we consider the case when $x,x+1,w,y,z$ all are $\neq0$. The
requirement that $z\rightarrow x\dashrightarrow x+1\rightarrow y$ and
$y\rightarrow w$ are present in $Gr(\phi)$ implies that in this case the
product $a_{\phi_{1}}a_{\phi_{2}}\cdots a_{\phi_{n}}$ contains sub-products
$a_{w}a_{y}a_{x+1}$ and $a_{x}a_{z}$. These two sub-products can
\textquotedblleft overlap" if and only if $z=w$.

If the sub-product $a_{w}a_{y}a_{x+1}$ appears before $a_{x}a_{z}$, then
\[
Pr(a_{1}a_{2}\cdots a_{n}=a_{\phi_{1}}a_{\phi_{2}}\cdots a_{\phi_{n}})
\]
can be written as
\begin{equation}
Pr(a_{1}\cdots a_{x}a_{y}^{-1}a_{x+1}^{\prime}a_{x+2}\cdots a_{n}=a_{\phi_{1}%
}\cdots a_{\phi_{r}}a_{w}a_{x+1}^{\prime}a_{\phi_{r+4}}\cdots a_{\phi_{s}%
}a_{x}a_{z}a_{\phi_{s+3}}\cdots a_{\phi_{n}}), \label{Case1a.Eq1}%
\end{equation}
where $a_{x+1}^{\prime}=a_{y}a_{x+1}$. Since both $y\rightarrow
x\dashrightarrow x+1\rightarrow w$ and $z\rightarrow y$ are present in
$Gr(\theta)$,
\[
Pr(a_{1}a_{2}\cdots a_{n}=a_{\theta_{1}}a_{\theta_{2}}\cdots a_{\theta_{n}})
\]
can be written as%
\begin{equation}
Pr(a_{1}\cdots a_{x}^{\prime}a_{y}^{-1}a_{x+1}a_{x+2}\cdots a_{n}%
=a_{\theta_{1}}\cdots a_{\theta_{r}}a_{w}a_{x+1}a_{\theta_{r+3}}\cdots
a_{\theta_{s-1}}a_{x}^{\prime}a_{z}a_{\theta_{s+3}}\cdots a_{\theta_{n}}),
\label{Case1a.Eq2}%
\end{equation}
where $a_{x}^{\prime}=a_{x}a_{y}$. Since $a_{\phi_{i}}=a_{\theta_{i}}$, for
$1\leq i\leq r$ and $s+3\leq i\leq n$, and $a\phi_{i}=a_{\theta_{i-1}}$, for
$r+4\leq i\leq s$, Equation \ref{Case1a.Eq1} and Equation \ref{Case1a.Eq2}
have the same number of ordered $n$-tuples of elements of $G$ as solutions.
Hence, we obtain the statement of our Theorem.

If the sub-product $a_{x}a_{z}$ appears before $a_{w}a_{y}a_{x+1}$, then
\[
Pr(a_{1}a_{2}\cdots a_{n}=a_{\phi_{1}}a_{\phi_{2}}\cdots a_{\phi_{n}})
\]
can be written as
\begin{equation}
Pr(a_{1}\cdots a_{x}a_{y}^{-1}a_{x+1}^{\prime}a_{x+2}\cdots a_{n}=a_{\phi_{1}%
}\cdots a_{\phi_{s}}a_{x}a_{z}a_{\phi_{s+3}}\cdots a_{\phi_{r}}a_{w}%
a_{x+1}^{\prime}a_{\phi_{r+4}}\cdots a_{\phi_{n}}), \label{Case1b.Eq1}%
\end{equation}
where $a_{x+1}^{\prime}=a_{y}a_{x+1}$. Since both $y\rightarrow
x\dashrightarrow x+1\rightarrow w$ and $z\rightarrow y$ are present in
$Gr(\theta)$,
\[
Pr(a_{1}a_{2}\cdots a_{n}=a_{\theta_{1}}a_{\theta_{2}}\cdots a_{\theta_{n}})
\]
can be written as%
\begin{equation}
Pr(a_{1}\cdots a_{x}^{\prime}a_{y}^{-1}a_{x+1}a_{x+2}\cdots a_{n}%
=a_{\theta_{1}}\cdots a_{\theta_{s}}a_{x}^{\prime}a_{z}a_{\theta_{s+4}}\cdots
a_{\theta_{r}}a_{w}a_{x+1}a_{\theta_{r+3}}\cdots a_{\theta_{n}}),
\label{Case1b.Eq2}%
\end{equation}
where $a_{x}^{\prime}=a_{x}a_{y}$. Since $a_{\phi_{i}}=a_{\theta_{i}}$, for
$1\leq i\leq r$ and $s+3\leq i\leq n$, and $a\phi_{i}=a_{\theta_{i-1}}$, for
$r+4\leq i\leq s$, Equation \ref{Case1b.Eq1} and Equation \ref{Case1b.Eq2}
have the same number of ordered $n$-tuples of elements of $G$ as their
solutions. Hence, we obtain the statement of our Theorem.

Second, we consider the case when $x=0$. In this case $x+1=1$, $z=\phi
_{1}=\theta_{1}$ and $y=\theta_{1}$. Now,
\[
Pr_{\phi}(G)=Pr(a_{1}a_{2}\cdots a_{n}=a_{z}a_{\phi_{2}}a_{\phi_{3}}\cdots
a_{\phi_{r}}a_{w}a_{y}a_{1}a_{\phi_{r+4}}\cdots a_{\phi_{n}})
\]
is equivalent to
\[
Pr(a_{y}a_{1}a_{2}\cdots a_{z-1}a_{y}{}^{-1}a_{y}a_{z}a_{z+1}\cdots
a_{n}=a_{y}a_{z}a_{\phi_{2}}a_{\phi_{3}}\cdots a_{\phi_{r}}a_{w}a_{y}%
a_{1}a_{\phi_{r+4}}\cdots a_{\phi_{n}}),
\]
which can be written as%
\begin{equation}
Pr(a_{1}^{\prime}a_{2}\cdots a_{n}=a_{y}a_{z}a_{\phi_{2}}a_{\phi_{3}}\cdots
a_{\phi_{r}}a_{w}a_{1}^{\prime}a_{\phi_{r+4}}\cdots a_{\phi_{n}}),
\label{Case2.Eq1}%
\end{equation}
where $a_{1}^{\prime}=a_{y}a_{1}$.

On the other hand,
\[
Pr_{\theta}(G)=Pr(a_{1}a_{2}\cdots a_{n}=a_{y}a_{z}a_{\theta_{3}}a_{\theta
_{4}}\cdots a_{\theta_{r+1}}a_{w}a_{1}a_{\theta_{r+4}}\cdots a_{\theta_{n}})
\]
is equivalent to
\begin{equation}
Pr(a_{1}a_{2}\cdots a_{n}=a_{y}a_{z}a_{\theta_{3}}a_{\theta_{4}}\cdots
a_{\theta_{r+1}}a_{w}a_{1}a_{\theta_{r+4}}\cdots a_{\theta_{n}}).
\label{Case2.Eq2}%
\end{equation}
Since $\theta_{3}=\phi_{2},...,\theta_{r+1}=\phi_{r}$ and $\theta_{r+4}%
=\phi_{r+4},...,\theta_{n}=\phi_{n}$, Equation \ref{Case2.Eq1} and Equation
\ref{Case2.Eq2} have the same number of ordered $n$-tuples of elements of $G$
as their solutions. Hence, we obtain the statement of our Theorem.

Third, we consider the case when $x=n$. In that case $x+1=0\operatorname{mod}%
n$, $y=\phi_{n}$ and $w=\phi_{n-1}=\theta_{n}$. Now, we see that
\[
Pr(a_{1}a_{2}\cdots a_{n}=a_{\phi_{1}}a_{\phi_{2}}\cdots a_{\phi_{n}})
\]
is equivalent to%
\begin{equation}
Pr(a_{1}a_{2}\cdots a_{n}=a_{\phi_{1}}a_{\phi_{2}}\cdots a_{\phi_{r}}%
a_{n}a_{z}a_{\phi_{r+3}}\cdots a_{\phi_{n-2}}a_{w}a_{y}), \label{Case3.Eq1}%
\end{equation}
On the other hand,
\[
Pr(a_{1}a_{2}\cdots a_{n}=a_{\theta_{1}}a_{\theta_{2}}\cdots a_{\theta_{n}})
\]
is equivalent to
\[
Pr(a_{1}a_{2}\cdots a_{n}a_{y}=a_{\theta_{1}}a_{\theta_{2}}\cdots
a_{\theta_{r}}a_{n}a_{y}a_{z}a_{\theta_{r+4}}\cdots a_{\theta_{n-1}}a_{w}%
a_{y}),
\]
which can be written as
\begin{equation}
Pr(a_{1}a_{2}\cdots a_{n}^{\prime}=a_{\theta_{1}}a_{\theta_{2}}\cdots
a_{\theta_{r}}a_{n}^{\prime}a_{z}a_{\theta_{r+4}}\cdots a_{\theta_{n-1}}%
a_{w}a_{y}), \label{Case3.Eq2}%
\end{equation}
where $a_{n}^{\prime}=a_{n}a_{y}$. Since $\theta_{1}=\phi_{1},...,\theta
_{r})=\phi_{r}$ and $\theta_{r+4}=\phi_{r+3},...,\theta_{n-1}=\phi_{n-2}$,
Equation \ref{Case3.Eq1} and Equation \ref{Case3.Eq2} have the same number of
ordered $n$-tuples of elements of $G$ as their solutions. Hence, we obtain the
statement of our Theorem.

Fourth, we consider the case when $y=0$. In that case $\phi_{1}=x+1$ and
$w=\phi_{n}$. Now, we see that
\[
Pr(a_{1}a_{2}\cdots a_{n}=a_{\phi_{1}}a_{\phi_{2}}\cdots a_{\phi_{n}})
\]
is equivalent to
\[
Pr(a_{1}a_{2}\cdots a_{n}=a_{x+1}a_{\phi_{2}}a_{\phi_{3}}\cdots a_{\phi_{r}%
}a_{x}a_{z}a_{\phi_{r+3}}\cdots a_{\phi_{n-1}}a_{w}),
\]
which can be written as:
\[
Pr(arsb=svrt),
\]
in which
\begin{align*}
r  &  =a_{x},s=a_{x+1},a=a_{1}a_{2}\cdots a_{x-1},\\
b  &  =a_{x+2}a_{x+3}\cdots a_{n},v=a_{\phi_{2}}a_{\phi_{3}}\cdots a_{\phi
_{i}},t=a_{\phi_{i+2}}a_{\phi_{i+3}}\cdots a_{\phi_{n}}.
\end{align*}
Here $i$ is such that $\phi_{i+1}=x$. On the other hand,
\[
Pr(a_{1}a_{2}\cdots a_{n}=a_{\theta_{1}}a_{\theta_{2}}\cdots a_{\theta_{n}})
\]
is equivalent to
\[
Pr(a_{1}a_{2}\cdots a_{n}=a_{z}a_{\theta_{2}}a_{\theta_{3}}a_{\theta_{4}%
}\cdots a_{\theta_{s}}a_{w}a_{x+1}a_{\theta_{r+3}}\cdots a_{\theta_{n-1}}%
a_{x}),
\]
which can be written as
\[
Pr(arsb=tsvr),
\]
which, in its turn, is equivalent to
\[
Pr(b^{-1}s^{-1}r^{-1}a^{-1}=r^{-1}v^{-1}s^{-1}t^{-1}).
\]
Now define
\[
a^{\prime}=b^{-1},b^{\prime}=a^{-1},v^{\prime}=v^{-1},t^{\prime}%
=t^{-1},r^{\prime}=s^{-1},\text{ and }s^{\prime}=r^{-1}.
\]
Then
\[
Pr(b^{-1}s^{-1}r^{-1}a^{-1}=r^{-1}v^{-1}s^{-1}t^{-1})=Pr(a^{\prime}r^{\prime
}s^{\prime}b^{\prime}=s^{\prime}v^{\prime}r^{\prime}t^{\prime})
\]
is the same as $Pr(arsb=tsvr)$. Hence, we obtain the statement of our Theorem.
\end{proof}

\begin{theorem}
\label{cyclic} Let $\phi$ and $\theta$ be two permutations in $S_{n}$, such
that $\theta$ is obtained from $\phi$ by an $x--y$ cyclic operation. Then%
\[
Pr(a_{1}a_{2}\cdots a_{n}=a_{\phi_{1}}a_{\phi_{2}}\cdots a_{\phi_{n}%
})=Pr(a_{1}a_{2}\cdots a_{n}=a_{\theta_{1}}a_{\theta_{2}}\cdots a_{\theta_{n}%
}).
\]

\end{theorem}

\begin{proof}
We consider four different possible cases, and prove our theorem for each one
of them:

First case is when $x\neq0,n$ and $y\neq0$. In this case, for some $1\leq
k\leq n-2$, $\phi_{k}=x$, $\phi_{k+1}=y$ and $\phi_{k+2}=x+1$. Hence,
\[
Pr(a_{1}a_{2}\cdots a_{n}=a_{\phi_{1}}a_{\phi_{2}}\cdots a_{\phi_{n}})
\]
can be written as
\begin{equation}
Pr(a_{1}a_{2}\cdots a_{n}=a_{\phi_{1}}a_{\phi_{2}}\cdots a_{\phi_{k-1}}%
a_{x}a_{y}a_{x+1}a_{\phi_{k+3}}\cdots a_{\phi_{n}}) \label{cyclic.first}%
\end{equation}

If $y>x$, then Expression \ref{cyclic.first} is equivalent to the expression
\begin{equation}
Pr(a_{1}\cdots(a_{x}a_{x+1})\cdots a_{y}\cdots a_{n}=a_{\phi_{1}}\cdots
a_{\phi_{k-1}}(a_{x}a_{x+1})(a_{x+1}^{-1}a_{y}a_{x+1})a_{\phi_{k+3}}\cdots
a_{\phi_{n}}) \label{cyclic.first.case1.eq1}%
\end{equation}

Now, we define:

\begin{itemize}
\item $\hat{a}_{t}=a_{t}$ if $t<x$ or $t>y$;

\item $\hat{a}_{x}=a_{x}a_{x+1}$;

\item $\hat{a}_{t}=a_{t+1}$ if $x+1<t<y-1$;

\item $\hat{a}_{y-1}=a_{x+1}$;

\item $\hat{a}_{y}=a_{x+1}^{-1}a_{y}$.
\end{itemize}

From the last two items we get that:

\begin{itemize}
\item $\hat{a}_{y-1}\hat{a}_{y}=a_{y}$;

\item $\hat{a}_{y}\hat{a}_{y-1}=a_{x+1}^{-1}a_{y}a_{x+1}$.
\end{itemize}

We see that Expression \ref{cyclic.first.case1.eq1} is equivalent to the
expression
\begin{equation}
Pr(\hat{a}_{1}\hat{a}_{2}\cdots\hat{a}_{n}=\hat{a}_{\theta_{1}}\hat{a}%
_{\theta_{2}}\cdots\hat{a}_{\theta_{k-1}}\hat{a}_{x}\hat{a}_{y}\hat{a}%
_{y-1}\hat{a}_{\theta_{k+3}}\cdots\hat{a}_{\theta_{n}}),
\label{cyclic.first.case1.eq2}%
\end{equation}

in which:

\begin{itemize}
\item $\theta_{t}=\phi_{t}$ for $\phi_{t}\leq x$ or $\phi_{t}\geq y$;

\item $\theta_{t}=y-1$ for $\phi_{t}=x+1$;

\item $\theta_{t}=\phi_{t}-1$ for $x+1<\phi_{t}<y$.
\end{itemize}

If $y<x$, then Expression \ref{cyclic.first} is equivalent to the expression%
\begin{equation}
Pr(a_{1}\cdots a_{y}\cdots(a_{x}a_{x+1})\cdots a_{n}=a_{\phi_{1}}\cdots
a_{\phi_{k-1}}(a_{x}a_{y}a_{x}^{-1})(a_{x}a_{x+1})a_{\phi_{k+3}}\cdots
a_{\phi_{n}}) \label{cyclic.first.case2.eq1}%
\end{equation}

Now, we define:

\begin{itemize}
\item $\hat{a}_{t}=a_{t}$ for $t<y$ or for $t>x+1$;

\item $\hat{a}_{t}=a_{t-1}$ for $y+1<t\leq x$;

\item $\hat{a}_{x+1}=a_{x}a_{x+1}$;

\item $\hat{a}_{y}=a_{y}a_{x-1}$;

\item $\hat{a}_{y+1}=a_{x}$.
\end{itemize}

From the last two items we get that:

\begin{itemize}
\item $\hat{a}_{y}\hat{a}_{y+1}=a_{y}$;

\item $\hat{a}_{y+1}\hat{a}_{y}=a_{x}a_{y}a_{x}^{-1}$.
\end{itemize}

We see that Expression \ref{cyclic.first.case2.eq1} is equivalent to the
expression
\begin{equation}
Pr(\hat{a}_{1}\hat{a}_{2}\cdots\hat{a}_{n}=\hat{a}_{\theta_{1}}\hat{a}%
_{\theta_{2}}\cdots\hat{a}_{\theta_{k-1}}\hat{a}_{y+1}\hat{a}_{y}\hat{a}%
_{x+1}\hat{a}_{\theta_{k+3}}\cdots\hat{a}_{\theta_{n}}),
\label{cyclic.first.case2.eq2}%
\end{equation}

in which:

\begin{itemize}
\item $\theta_{t}=\phi_{t}$ for $\phi_{t}\leq y$ or $\phi_{t}\geq x+1$;

\item $\theta_{t}=y+1$ for $\phi_{t}=x$;

\item $\theta_{t}=\phi_{t}+1$ for $y<\phi_{t}<x$.
\end{itemize}

Second case is when $x=0$. In this case, $\phi_{1}=y$ and $\phi_{2}=1$.
Hence,
\[
Pr(a_{1}a_{2}\cdots a_{n}=a_{\phi_{1}}a_{\phi_{2}}\cdots a_{\phi_{n}})
\]
can be written as
\begin{equation}
Pr(a_{1}a_{2}\cdots a_{n}=a_{y}a_{1}a_{\phi_{3}}\cdots a_{\phi_{n}})
\label{cyclic.second}%
\end{equation}

Expression \ref{cyclic.second} is equivalent to the expression
\begin{equation}
Pr(a_{2}\cdots a_{n}=(a_{1}^{-1}a_{y}a_{1})a_{\phi_{3}}\cdots a_{\phi_{n}})
\label{cyclic.second1}%
\end{equation}

Now, we define:

\begin{itemize}
\item $\hat{a}_{t}=a_{t}$ for $t>y$;

\item $\hat{a}_{t}=a_{t+1}$ for $1\leq t<y-1$;

\item $\hat{a}_{y-1}=a_{1}$;

\item $\hat{a}_{y}=a_{1}^{-1}a_{y}$.
\end{itemize}

From the last two items we get that:

\begin{itemize}
\item $\hat{a}_{y-1}\hat{a}_{y}=a_{y}$;

\item $\hat{a}_{y}\hat{a}_{y-1}=a_{1}^{-1}a_{y}a_{1}$.
\end{itemize}

We see that Expression \ref{cyclic.second1} is equivalent to the expression%
\begin{equation}
Pr(\hat{a}_{1}\cdots\hat{a}_{n}=\hat{a}_{y}\hat{a}_{y-1}\hat{a}_{\theta_{3}%
}\cdots\hat{a}_{\theta_{n}}), \label{cyclic.second2}%
\end{equation}

in which:

\begin{itemize}
\item $\theta_{t}=\phi_{t}$ for $\phi_{t}\geq y$;

\item $\theta_{t}=y-1$ for $\phi_{t}=1$;

\item $\theta_{t}=\phi_{t}-1$ for $1<\phi_{t}<y$.
\end{itemize}

Third case $-$ when $x=n$. In this case, $\phi_{n}=y$ and $\phi_{n-1}=n$.
Hence,
\[
Pr(a_{1}a_{2}\cdots a_{n}=a_{\phi_{1}}a_{\phi_{2}}\cdots a_{\phi_{n}})
\]
can be written as
\begin{equation}
Pr(a_{1}a_{2}\cdots a_{n}=a_{\phi_{1}}\cdots a_{\phi_{n-2}}a_{n}a_{y})
\label{cyclic.third}%
\end{equation}

Expression \ref{cyclic.third} is equivalent to the expression
\begin{equation}
Pr(a_{1}a_{2}\cdots a_{n-1}=a_{\phi_{1}}\cdots a_{\phi_{n-2}}a_{n}a_{y}%
a_{n}^{-1}) \label{cyclic.third1}%
\end{equation}

Now, we define:

\begin{itemize}
\item $\hat{a}_{t}=a_{t}$ for $t<y$;

\item $\hat{a}_{t}=a_{t-1}$ for $y+1<t<n$;

\item $\hat{a}_{y}=a_{y}a_{n}^{-1}$;

\item $\hat{a}_{y+1}=a_{n}$.
\end{itemize}

From the last two items we get that:

\begin{itemize}
\item $\hat{a}_{y}\hat{a}_{y+1}=a_{y}$;

\item $\hat{a}_{y+1}\hat{a}_{y}=a_{n}a_{y}a_{n}^{-1}$.
\end{itemize}

We see that Expression \ref{cyclic.third1} is equivalent to the expression
\begin{equation}
Pr(\hat{a}_{1}\cdots\hat{a}_{n}=\hat{a}_{\theta_{1}}\cdots\hat{a}%
_{\theta_{n-2}}\hat{a}_{y+1}\hat{a}_{y}), \label{cyclic.third2}%
\end{equation}

in which:

\begin{itemize}
\item $\theta_{t}=\phi_{t}$ for $1\leq\phi_{t}\leq y$;

\item $\theta_{t}=\phi_{t}+1$ for $y<\phi_{t}<n$;

\item $\theta_{t}=y+1$ for $\phi_{t}=n$.
\end{itemize}

Forth case is when $y=0$. In this case, $\phi_{1}=x+1$ and $\phi_{n}=x$.
Hence,
\[
Pr(a_{1}a_{2}\cdots a_{n}=a_{\phi_{1}}a_{\phi_{2}}\cdots a_{\phi_{n}})
\]
can be written as
\begin{equation}
Pr(a_{1}a_{2}\cdots a_{n}=a_{x+1}a_{\phi_{2}}\cdots a_{\phi_{n-1}}a_{x})
\label{cyclic.forth}%
\end{equation}

Expression \ref{cyclic.forth} is equivalent to the expression
\begin{equation}
Pr(a_{x}a_{1}a_{2}\cdots a_{x-1}(a_{x}a_{x+1})\cdots a_{n}a_{x}^{-1}%
=(a_{x}a_{x+1})a_{\phi_{2}}\cdots a_{\phi_{n-1}}) \label{cyclic.forth1}%
\end{equation}

In its turn, Expression \ref{cyclic.forth1} is equivalent to the expression
\begin{equation}
Pr(a_{x}a_{1}a_{2}\cdots a_{x-1}(a_{x}a_{x+1})\cdots a_{n}=(a_{x}%
a_{x+1})a_{\phi_{2}}\cdots a_{\phi_{n-1}}a_{x}) \label{cyclic.forth2}%
\end{equation}

Now, we define:

\begin{itemize}
\item $\hat{a}_{t}=a_{t}$ for $t>x+1$;

\item $\hat{a}_{1}=a_{x}$;

\item $\hat{a}_{t}=a_{t-1}$ for $1<t\leq x$;

\item $\hat{a}_{x+1}=a_{x}a_{x+1}$.
\end{itemize}

We see that Expression \ref{cyclic.forth2} is equivalent to the expression
\begin{equation}
Pr(\hat{a}_{1}\cdots\hat{a}_{n}=\hat{a}_{x+1}\hat{a}_{\theta}{}_{2}\cdots
\hat{a}_{\theta}{}_{n-1}\hat{a}_{1}), \label{cyclic.forth3}%
\end{equation}

in which:

\begin{itemize}
\item $\theta_{t}=\phi_{t}$ for $\phi_{t}>x$;

\item $\theta_{t}=1$ for $\phi_{t}=x$;

\item $\theta_{t}=\phi_{t}+1$ for $1\leq\phi_{t}<x$.
\end{itemize}
\end{proof}

\begin{example}
\label{spec.example}

Let, again, the group $G$ be $D_{4}$ or $Q_{8}$. We get that%
\[
Pr(a_{1}a_{2}a_{3}a_{4}\cdots a_{2k-1}a_{2k}a_{2k+1}a_{2k+2}\cdots a_{n}%
=a_{2}a_{1}a_{4}a_{3}\cdots a_{2k}a_{2k-1}a_{2k+1}a_{2k+2}\cdots a_{n}%
\]
is equal to%
\[
(\frac{5}{8})^{k}+\frac{k!}{2!(k-2)!}(\frac{5}{8})^{k-2}(\frac{3}{8}%
)^{2}+\frac{k!}{4!(k-4)!}(\frac{5}{8})^{k-4}(\frac{3}{8})^{4}+...
\]

Since an even number of the inverted pairs $a_{j+1}a_{j}$ must be equal to
$ca_{j}a_{j+1}$, where $c$ is the nontrivial element from the center of $G$.
By a direct computation we see that for $k=0,1,2,...[[\frac{n}{2}]]$ we get
different probabilities for the corresponding permutation equalities.
\end{example}

\begin{theorem}
\label{equiv-alternating} If $\theta$ and $\phi$ are $x--y$ equivalent then
$Gr(\phi)$ and $Gr(\theta)$ have the same number of alternating cycles.
\end{theorem}

\begin{proof}
It is sufficient to prove that if $\theta$ is obtained from $\phi$ by an
$x--y$ exchange operation or by an $x--y$ cyclic operation, then $Gr(\phi)$
and $Gr(\theta)$ have the same number of alternating cycles.

If $\theta$ is obtained from $\phi$ by an $x--y$ exchange operation, then the
big black\newline$(n+1)$-cycle $\phi^{\cdot}$ contains $z\rightarrow x$ and
$x+1\rightarrow y\rightarrow w$, and the the big black $(n+1)$-cycle
$\theta^{\cdot}$ contains $z\rightarrow y\rightarrow x$ and $x+1\rightarrow
w$. Thus, (see Definition \ref{def.exchange}) the only difference between big
black $(n+1)$-cycles $\phi^{\cdot}$ and $\theta^{\cdot}$ is the relocation of
$y$ from being immediately after $x+1$ in $\phi^{\cdot}$ to being immediately
before $x$ in $\theta^{\cdot}$. Since, $\phi^{\circ}=\phi^{\cdot}%
\cdot(0,1,...,n)$, and $\theta^{\circ}=\theta^{\cdot}\cdot(0,1,...,n)$, we get%
\[
\phi^{\circ}=(....,x+1,y,w,...,z,x,...)\cdot(0,1,...,x,x+1,...,n)
\]
and%
\[
\theta^{\circ}=(....,x+1,w,...,z,y,x)\cdot(0,1,...,y-1,y,...,n).
\]
This implies that $\phi^{\circ}(x)=\phi^{\cdot}(x+1)=y$ and $\theta^{\circ
}(y-1)=\theta^{\cdot}(y)=x$. Therefore, the only difference between
$\theta^{\circ}$ and $\phi^{\circ}$ is the relocation of $x$ from being
immediately before $y$ in the cyclic presentation of $\phi^{\circ}$ to being
immediately after $y-1$ in the cyclic presentation of $\theta^{\circ}$. Hence,
the number of cycles in $\theta^{\circ}$, as we count the cycles of length one
(fixed elements), is the same as the number of cycles in $\phi^{\circ}$.

If $\theta$ is obtained from $\phi$ by an $x--y$ cyclic operation then it is
clear from Lemma \ref{howcyclicworks} that the cycles of $\theta^{\circ}$ are
obtained from the cycles of $\phi^{\circ}$ by certain trivial interchanges.
Hence, the number of cycles in $\theta^{\circ}$ and in $\phi^{\circ}$, as we
count the cycles of length one, is the same.

Therefore, after performing $x--y$ exchange and $x--y$ cyclic operations on
$\phi$ we obtain $\theta$, such that the number of cycles in $\theta^{\circ}$
is the same as in $\phi^{\circ}$. And the numbers of cycles in $\phi^{\circ}$
and $\theta^{\circ}$ are equal to the numbers of alternating cycles in the
cycle graphs $Gr(\phi)$ and $Gr(\theta)$, respectively.
\end{proof}

To prove Theorem \ref{hultman-exchange-cyclic} below, which asserts the
opposite direction of Theorem \ref{equiv-alternating}, we need the following
four technical lemmas.

\begin{lemma}
\label{tech.lem.1} Let $\phi$ and $\theta$ be $x--y$ equivalent permutations
in $S_{2t}$, such that $Gr(\phi)$ and, consequently, $Gr(\theta)$ contain only
one alternating cycle. Let $\phi^{\cdot}=(\phi_{2t},\phi_{2t-1},\dots,\phi
_{0})$ and $\theta^{\cdot}=(\theta_{2t},\theta_{2t-1},\dots,\theta_{0})$. Then
any two permutations $\mu,\tau\in S_{2t+2}$, such that%
\[
\mu^{\cdot}=(\phi_{i},\phi_{i-1},\dots,\phi_{0},\phi_{2t},\dots,\phi
_{i+1},2t+1,2t+2)
\]
and%
\[
\tau^{\cdot}=(\theta_{j},\theta_{j-1},\dots,\theta_{0},\theta_{2t}%
,\dots,\theta_{j+1},2t+1,2t+2),
\]
are also $x--y$ equivalent and also have only one alternating cycle in their
cycle graphs.
\end{lemma}

\begin{proof}
The fact that $Gr(\phi)$ contains only one alternating cycle implies that
$\phi^{\circ}=\phi^{\cdot}\cdot(0,1,\dots,2t)$ is a cycle of length $2t+1$ in
$S(1+2t)$. Consider $\mu\in S_{2t+2}$ such that
\[
\mu^{\cdot}=(\phi_{2t},\phi_{2t-1},\dots,\phi_{0},2t+1,2t+2).
\]
In $\phi^{\circ}$ the element $\phi_{u}=\phi_{0}-1$ (where $-1$ means $2t$)
went to element $\phi_{2t}$ and the element $\phi_{v}=2t$ went to element
$\phi^{\cdot}(0)$. In $\mu^{\circ}$ the element $\phi_{u}=\phi_{0}-1$ goes to
element $2t+1$ which, in its turn, goes to element $\phi_{2t}$, and the
element $\phi_{v}=2t$ goes to element $2t+2$ which, in its turn, goes to
element $\phi^{\cdot}(0)$. This implies that $\mu^{\circ}$ is just a cycle of
length $2t+3$. Hence, $Gr(\mu)$ also contains only one alternating
cycle.\newline

First, let us show that the permutation $\mu\in S_{2t+2}$ is equivalent to
some permutation $\tau$, such that
\[
\tau^{\cdot}=(\theta_{j},\theta_{j-1},\dots,\theta_{0},\theta_{2t}%
,\dots,\theta_{j+1},2t+1,2t+2).
\]

We start by considering the situation, when $\theta$ was obtained from $\phi$
by one $x--y$ exchange operation.

In the case, when $x\neq2t$ (note, that $2t=0-1=\phi_{0}-1\operatorname{mod}%
2t$) and $y\neq\phi_{2t}$, performing the same $x--y$ exchange operation on
$\phi$ and on $\mu$ produces two permutations $\theta$ and $\tau$,
respectively, such that $\tau^{\cdot}$ is just $\theta^{\cdot}$ with
$2t+1\rightarrow2t+2$ inserted in some place. Performing that $x--y$ exchange
operation on $\phi$ relocates $y$ from its place immediately after $x+1$ in
$\phi^{\cdot}$ to its place immediately before $x$ in $\theta^{\cdot}$.
Performing that same $x--y$ exchange operation on $\mu$ relocates $y$ from its
place immediately after $x+1$ in $\mu^{\cdot}$ to its place immediately before
$x$ in $\tau^{\cdot}$.

Now let's look at the permutation $\mu$, for which we have that $y\neq
\phi_{2t},2t+1,2t+2$, $x\neq2t,2t+1,2t+2$ and $x+1\neq\phi_{0},2t+1,2t+2$. In
the big black cycle $\mu^{\cdot}$ the element $\phi_{2t}$ stands right after
the element $2t+2$, and the element $\phi_{0}$ stands right before the element
$2t+1$. Hence, the whole piece $\phi_{0}\rightarrow2t+1\rightarrow
2t+2\rightarrow\phi_{2t}$ of $\mu$ is not \textquotedblleft broken" by our
$x--y$ exchange operation, and is preserved in $\tau^{\cdot}$.

In the case, when $x=2t=\phi_{0}-1$, we get that $y=\phi_{2t}$. Since the
element $y$ in $\phi^{\cdot}$ has to be standing right after the element
$x+1=\phi_{0}$, we get that $y=\phi_{2t}$. Let $i$ be such that $\phi
_{i}=2t=x$. Keep in mind, that in all the descriptions of the black cycles,
which follow, $\phi_{2t}=y$, $\phi_{i}=2t=x$ and $x+1=\phi_{0}$. The $x--y$
exchange operation, applied on $\phi$, produces $\theta$, such that in
$\theta^{\cdot}$ the element $y=\phi_{2t}$ stands right before the element
$x=\phi_{0}-1=2t=\phi_{i}$. Thus, we have%
\[
\theta^{\cdot}=\phi_{2t-1}\rightarrow\phi_{2t-2}\rightarrow...\rightarrow
\phi_{i+1}\rightarrow y\rightarrow x\rightarrow\phi_{i-1}\rightarrow
...\rightarrow\phi_{0}\rightarrow\phi_{2t-1}.
\]

Performing $x--(2t+1)$ exchange operation on $\mu$ produces a permutation
$\rho$, such that in $\rho^{\cdot}$ the element $2t+1$ stands right before the
element $x=\phi_{0}-1$, and the element $2t+2$ stands right after the element
$x+1=\phi_{0}$.

Thus we have%
\[
\rho^{\cdot}=\phi_{2t}\rightarrow\phi_{2t-1}\rightarrow...\rightarrow
\phi_{i+1}\rightarrow(2t+1)\rightarrow x\rightarrow\phi_{i-1}\rightarrow
...\rightarrow(x+1)\rightarrow(2t+2)\rightarrow y.
\]

Next, performing $x--(2t+2)$ exchange operation on $\rho$ produces a
permutation $\varrho$, such that%
\[
\varrho^{\cdot}=\phi_{2t}\rightarrow\phi_{2t-1}\rightarrow...\rightarrow
\phi_{i+1}\rightarrow(2t+1)\rightarrow(2t+2)\rightarrow x\rightarrow\phi
_{i-1}\rightarrow...\rightarrow(x+1)\rightarrow y.
\]

Finally, performing $x--y$ exchange operation on $\varrho$ produces a
permutation $\tau$, such that%
\begin{gather*}
\tau^{\cdot}=\phi_{2t-1}\rightarrow\phi_{2t-2}\rightarrow...\rightarrow
\phi_{i+1}\rightarrow(2t+1)\rightarrow(2t+2)\rightarrow\\
\rightarrow y\rightarrow x\rightarrow\phi_{i-1}\rightarrow....\rightarrow
\phi_{0}\rightarrow\phi_{2t-1}.
\end{gather*}

Now we prove that any permutation $\pi$, such that
\[
\pi^{\cdot}=(\theta_{s},\theta_{s-1},\dots,\theta_{0},\theta_{2t},\dots
,\theta_{s+1},2t+1,2t+2)
\]
can be obtained from the permutation $\mu$ by several $x--y$ exchange
operations. We have already shown, that some $\tau$, such that
\[
\tau^{\cdot}=(\theta_{j},\theta_{j-1},\dots,\theta_{0},\theta_{2t}%
,\dots,\theta_{j+1},2t+1,2t+2),
\]
can be obtained from $\mu$. Performing $(2t+2)--\theta_{j}$ exchange operation
on $\tau$ produces permutation $\xi$, such that
\[
\xi^{\cdot}=(\theta_{j-1},\dots,\theta_{0},\theta_{2t},\dots,\theta
_{j+1},\theta_{j},2t+1,2t+2).
\]
Performing $(2t+2)--\theta_{j-1}$ exchange operation on $\xi$ produces
permutation $\chi$, such that
\[
\chi^{\cdot}=(\theta_{j-2},\dots,\theta_{0},\theta_{2t},\dots,\theta
_{j},\theta_{j-1},2t+1,2t+2).
\]
Continuing this process, we can obtain any permutation $\pi$, such that
\[
\pi^{\cdot}=(\theta_{s},\theta_{s-1},\dots,\theta_{0},\theta_{2t},\dots
,\theta_{s+1},2t+1,2t+2).\newline%
\]

In the situation, when $\theta$ was obtained from $\phi$ by one $x--y$ cyclic
operation, it is trivial to check that any $x--y$ cyclic operation, where $x$
and $y$ are between $0$ and $2t$, does not affect $2t+1\rightarrow2t+2$ in the
big black cycle $\mu^{\cdot}$.

All of the argument above, obviously, also works in the situations, when
$\phi$ was obtained from $\theta$ by one $x--y$ exchange or one $x--y$ cyclic operation.

Therefore, since any $x--y$ equivalence is obtained by performing several
$x--y$ exchange and $x--y$ cyclic operations, our Lemma holds.
\end{proof}

\begin{lemma}
\label{tech.lem.2} Let $\phi$ be a permutation in $S_{2t}$, such that
$Gr(\phi)$ contains only one alternating cycle. If
\[
\phi^{\cdot}=(\phi_{2t-2},\phi_{2t-1},\dots,\phi_{0},2t-1,2t)
\]
then for any $0\leq a\leq2t$ and any permutation $\theta$ in $S_{2t}$, such
that
\[
\theta^{\cdot}=(\theta_{2t-2},\theta_{2t-1},\dots,\theta_{0},a,a+1)
\]
(here $a+1$ is considered modulo $(2t+1)$), where $\theta_{i}=\phi
_{i}+a\operatorname{mod}\left(  2t+1\right)  $ for all $i$, $\phi$ and
$\theta$ are in the same $x--y$ exchange orbit. Therefore, $\phi$ and $\theta$
are $x--y$ equivalent.
\end{lemma}

\begin{proof}
Performing the $(2t-1)--(2t)$ exchange operation on $\phi$ produces $\rho$
such that in $\rho^{\cdot}$ the element $2t$, which in $\phi^{\cdot}$ stood
right after $2t-1$, now stands right before $2t-2$. Next, performing the
$(2t)--(2t-2)$ exchange operation on $\rho$ produces $\varrho$ such that in
$\varrho^{\cdot}$ the element $2t-2$, which in $\rho^{\cdot}$ stood right
after $2t$, now stands right before $2t-1$. Thus, for the permutation
$\varrho$, produced by performing these two $x--y$ exchange operations, we
have
\[
\varrho^{\cdot}=(\varrho_{2t-2},\varrho_{2t-1},\dots,\varrho_{0},2t-2,2t-1).
\]
Here $\varrho_{j}=\phi_{j}$, when $\phi_{j}\neq2t-2$, and $\varrho_{j}%
=2t=\phi_{j}+2$, when $\phi_{j}=2t-2$.

Now, we have $(2t-2)\rightarrow(2t-1)$ in $\varrho^{\cdot}$. We proceed as
follows. Perform $(2t-2)--(2t-1)$ exchange operation on $\varrho$ to produce
$\xi$, such that in $\xi^{\cdot}$ the element $2t-1$, which in $\varrho
^{\cdot}$ stood right after $2t-2$, now stands right before $2t-3$. Next
perform $(2t-1)--(2t-3)$ exchange operation on $\xi$ to produce $\chi$, such
that in $\chi^{\cdot}$ the element $2t-3$, which in $\xi^{\cdot}$ stood right
after $2t-1$, now stands right before $2t-2$. Notice, that
\[
\chi^{\cdot}=(\chi_{2t-2},\chi_{2t-1},\dots,\chi_{0},2t-3,2t-2),
\]
where $\chi_{j}=\phi_{j}$, when $\phi_{j}\neq2t-3,2t-2$, and $\chi_{j}%
=\phi_{j}+2$, when $\phi_{j}=2t-3$ or $\phi_{j}=2t-2$.

We proceed this way, until we obtain a permutation $\pi$, such that
\[
\pi^{\cdot}=(\pi_{2t-2},\pi_{2t-1},\dots,\pi_{0},0,1)
\]
where $\pi_{j}=\phi_{j}+2\operatorname{mod}\left(  2t+1\right)  $. Since,
$2t+1$ is an odd number, for any $0\leq a\leq2t$, this whole process can now
be repeated several times to obtain permutation $\theta$, such that
\[
\theta^{\cdot}=(\theta_{2t-2},\theta_{2t-1},\dots,\theta_{0},a-1,a),
\]
where $\theta_{j}=\phi_{j}+a+1\operatorname{mod}\left(  2t+1\right)  $.
\end{proof}

\begin{lemma}
\label{tech.lem.3} Let $\phi$ be a permutation in $S_{2t}$, such that
$Gr(\phi)$ contains only one alternating cycle and that $\phi^{\cdot
}=(...,a,b,a-1,...)$, where $a$ and $b$ are any numbers between $0$ and $2t$
(again, $-1$ means $2t$). Then permutation $\phi$ is in the same $x--y$ orbit
as some permutation $\theta$ in $S_{2t}$, such that
\[
\theta^{\cdot}=(\theta_{2t-2},\theta_{2t-1},\dots,\theta_{0},2t-1,2t).
\]

\end{lemma}

\begin{proof}
Performing $(a-1)--b$ cyclic operation on $\phi$ produces permutation $\rho$,
such that if $b>a-1$, then
\[
\rho^{\cdot}=(...,b-1,b,a-1,...),
\]
and if $b<a-1$, then
\[
\rho^{\cdot}=(...,a,b,b+1,...).
\]
Repeating our arguments from the proof of Lemma \ref{tech.lem.2} establishes
the existence of the permutation $\theta$ in $S_{2t}$, such that
\[
\theta^{\cdot}=(\theta_{2t-2},\theta_{2t-1},\dots,\theta_{0},2t-1,2t).
\]

\end{proof}

\begin{lemma}
\label{tech.lem.4} Let $\phi$ be a permutation in $S_{2t}$, such that
$Gr(\phi)$ contains only one alternating cycle. By performing $x--y$ exchange
and $x--y$ cyclic operations on $\phi$, it is possible to obtain some
permutation $\theta$ in $S_{2t}$, such that
\[
\theta^{\cdot}=(\theta_{2t-2},\theta_{2t-1},\dots,\theta_{0},2t-1,2t).
\]

\end{lemma}

\begin{proof}
Select any number $a$ between $1$ and $2t$. Let $\phi^{\cdot}$ contain
\[
a\rightarrow b_{1}\rightarrow b_{2}\rightarrow\dots\rightarrow b_{k}%
\rightarrow(a-1),
\]
which means that there are $k$ elements between $a$ an $a-1$ in the big black
cycle $\phi^{\cdot}$. Clearly, $k>0$ (otherwise, in $Gr(\phi)$ the element
$a-1$ is taken by a grey arrow to the element $a$, which is then taken by a
black arrow back to $a-1$ $-$ contradiction to the fact that $Gr(\phi)$
contains only one alternating cycle). We want to show that if $k>2$, we can
decrease $k$ by performing $x--y$ exchange operations on $\phi$. Let
$b=\min\{b_{1},b_{2},\dots,b_{k}\}$. Consider three cases:

Case 1: $b=b_{i}$ for some $1\leq i\leq k-1$. Performing $b_{i}--b_{i+1}$
exchange operation on $\phi$ produces $\rho$, such that
\[
\rho^{\cdot}=(...,a,b_{1},b_{2},\dots,b_{i},b_{i+2},\dots,b_{k},a-1,...).
\]

Case 2: $b=b_{k}$ and $b_{k-1}-1=b_{i}\operatorname{mod}2t$ for some $1\leq
i\leq k-1$. Performing $b_{k-1}--b_{k}$ exchange operation on $\phi$ produces
$\rho$, such that
\[
\rho^{\cdot}=(...,a,b_{1},b_{2},\dots,b_{i-1},b_{k},b_{i},b_{i+1}%
,\dots,b_{k-1},a-1,...).
\]
Next, performing $b_{k}--b_{i}$ exchange operation on $\rho$ produces
$\varrho$, such that
\[
\varrho^{\cdot}=(...,a,b_{1},b_{2},\dots,b_{i-1},b_{k},b_{i+1},\dots
,b_{k-1},a-1,...).
\]

Case 3: $b=b_{k}$ and $b_{k-1}-1\neq b_{j}$ for any $1\leq j\leq k-1$.
Performing $b_{k-1}--b_{k}$ exchange operation on $\phi$ creates $\rho$, such
that
\[
\rho^{\cdot}=(...,a,b_{1},b_{2},\dots,b_{k-1},a-1,...).
\]

Thus in all three cases above we reduced the number of elements between $a$
and $a-1$ in $\phi\prime^{\cdot}$ by performing appropriate $x--y$ exchange
operations. Repeating this process again and again eventually produces a
permutation $\tau$, such that $\tau^{\cdot}=(...,a,d,a-1,...)$. Now, we apply
Lemma \ref{tech.lem.3} to $\tau$ and obtain the result of this lemma.
\end{proof}

\begin{theorem}
\label{hultman-exchange-cyclic} If two permutations $\phi$ and $\theta$ in
$S_{n}$ have the same number of alternating cycles in their cycle graphs
$Gr(\phi)$ and $Gr(\theta)$, then they are $x--y$ equivalent.
\end{theorem}

\begin{proof}
Recall, that it was already established in Theorem \ref{equiv-alternating}
that if $\theta$ and $\phi$ are $x--y$ equivalent then $Gr(\phi)$ and
$Gr(\theta)$ have the same number of alternating cycles. This theorem
establishes the other direction.

First, we prove our theorem for the case, when $n=2t$ and $\phi$ and $\theta$
have one alternating cycle in their cycle graphs $Gr(\phi)$ and $Gr(\theta)$.
If $t=1$ and $n=2t=2$, this statement is trivial, since there is only one
permutation in $S_{2}$, which has one alternating cycle in its cycle graph.

We proceed by induction on $n=2t$. Assume that any two permutations $\phi$ and
$\theta$ in $S_{2t}$, such that $Gr(\phi)$ and $Gr(\theta)$ contain one
alternating cycle, are $x--y$ equivalent. Let $\phi^{\cdot}=(\phi_{2t}%
,\phi_{2t-1},\dots,\phi_{0})$ and $\theta^{\cdot}=(\theta_{2t},\theta
_{2t-1},\dots,\theta_{0}).$

By Lemma \ref{tech.lem.1}, any two permutation $\mu$ and $\tau$ in $S_{2t+2}$,
such that%
\[
\mu^{\cdot}=(\phi_{i},\phi_{i-1},\dots,\phi_{0},\phi_{2t},\dots,\phi
_{i+1},2t+1,2t+2)
\]
and%
\[
\tau^{\cdot}=(\theta_{j},\theta_{j-1},\dots,\theta_{0},\theta_{2t}%
,\dots,\theta_{j+1},2t+1,2t+2),
\]
also have only one alternating cycle in their cycle graphs and are also $x--y$ equivalent.

Now, Lemma \ref{tech.lem.4} asserts that any permutation $\sigma$ in
$S_{2t+2}$, such that $Gr(\sigma)$ contains only one alternating cycle, is
$x--y$ equivalent to some permutation $\mu$ in $S_{2t+2}$, such that
$\mu^{\cdot}$ contains $(2t+1)\rightarrow(2t+2)$. From Theorem
\ref{equiv-alternating} we know, that $Gr(\mu)$ also contains only one
alternating cycle. Again, Lemma \ref{tech.lem.4} asserts that any permutation
$\pi$ in $S_{2t+2}$, such that $Gr(\pi)$ contains only one alternating cycle,
is $x--y$ equivalent to some permutation $\tau$ in $S_{2t+2}$, such that
$\tau^{\cdot}$ contains $(2t+1)\rightarrow(2t+2)$. Again, Theorem
\ref{equiv-alternating} asserts that $Gr(\tau)$ contains only one alternating
cycle. Since we have obtained above, using induction and Lemma
\ref{tech.lem.1}, that any such $\mu$ and $\tau$ in $S_{2t+2}$ are $x--y$
equivalent, we get that $\sigma$ and $\pi$ are also $x--y$ equivalent. This
completes the induction argument.

We prove our theorem by induction on $n$. The statement is obvious for $n=1$
and $n=2$. Now, assume that the theorem is true for all $n\leq k$. We prove it
for $n=k+1$.

We start by considering a permutation $\phi\in S_{k+1}$, which fixes the
element $k+1$. Thus, $\phi(k+1)=k+1$. This means, that $(k+1)$ is a cycle of
length one in the cyclic decomposition of $\phi$. This also means, that
$\phi^{\cdot}$ contains $0\rightarrow(k+1)$. Thus in $\phi^{\circ}$ the
element $k+1$ is also a fixed element, just as it is in $\phi$, and
$\phi^{\circ}$ also contains a cycle $(k+1)$ of length one. We shall now prove
that any permutation, which has the same number of alternating cycles in its
cycle graph as in $Gr(\phi)$, is $x--y$ equivalent to $\phi$. Let $\rho\in
S_{k}$ be $\phi$, with its fixed element $k+1$ deleted.

Notice, that performing on $\phi$ any $x--y$ exchange or $x--y$ cyclic
operation, where $x\neq k+1$ and $y\neq0,k+1$, leaves $0\rightarrow k+1$ in
the big black cycle $\phi^{\cdot}$ unbroken. Consequently, such operation does
not effect the cycle $(k+1)$ in $\phi^{\circ}$. Hence, performing any $x--y$
exchange and $x--y$ cyclic operation, with $y\neq0$, on $\rho$ is equivalent
to performing this same operation on $\phi$, and then deleting the fixed
element $k+1$ from the result.

Notice, that in order to be permissible to perform any $x--0$ cyclic
operation, $x$ must be less than $k+1$. Otherwise, $x+1=0\operatorname{mod}%
\left(  k+1\right)  $, hence $y=x+1$, which is impossible. Thus, any $x--0$
cyclic operation does not move $k+1$ from its place in $\phi^{\cdot}$. Hence,
$x--0$ cyclic operation preserves $k+1$ as a fixed element.

Notice also, that performing on $\phi$ first $x--0$ exchange operation, with
any $x\neq k+1$, and then $x--(k+1)$ exchange operation, will leave
$0\rightarrow k+1$ unbroken in the big black cycle $\phi^{\cdot}$.
Consequently, this composition of two exchange operations does not effect the
cycle $(k+1)$ in $\phi^{\circ}$. Hence, performing an $x--0$ exchange
operation on $\rho$ is equivalent to performing this same $x--0$ exchange
operation, followed by $x--(k+1)$ exchange operation, on $\phi$, and then
deleting the fixed element $k+1$ from the result.

First, consider $\theta\in S_{k+1}$, such that $Gr(\theta)$ has the same
number of alternating cycles as $Gr(\phi)$, and which, like $\phi$, fixes the
element $k+1$. Let $\varrho\in S_{k}$ be $\theta$, with its fixed element
$k+1$ deleted. By the induction hypothesis, we can obtain $\varrho$ from
$\rho$ by performing $x--y$ exchange and $x--y$ cyclic operations. But, due to
our above argument, this implies that we can also obtain $\theta$ from $\phi$
by performing $x--y$ exchange and $x--y$ cyclic operations.

Now we will show that any permutation $\theta\in S_{k+1}$, which has two or
more alternating cycles in its cycle graph $Gr(\theta)$, is $x--y$ equivalent
to some permutation $\phi\in S_{k+1}$, which fixes the element $k+1$. Let%
\[
cyc=(k+1)\mapsto\theta_{j_{2}}\mapsto...\mapsto\theta_{j_{m}}\mapsto(k+1)
\]
be the cycle of $\theta^{\circ}$, which contains the element $\theta_{j_{1}%
}=k+1$. If the length of $cyc$ is $1$, then just take $\phi=\theta$. If the
length of $cyc$ is $2$, then $cyc=(k+1)\mapsto\theta_{j_{2}}\mapsto(k+1)$.
Here $\theta_{j_{2}}\neq k$, since there can be no $i\mapsto(i+1)$, for any
$i$, in any cycle of $\theta^{\circ}$. Performing $\theta_{j_{2}}--(k+1)$
exchange operation on $\theta$ produces a permutation $\phi$, such that
$\phi^{\circ}$ contains $(k+1)\mapsto(k+1)$. Thus, $\theta$ is $x--y$
equivalent to $\phi$, which fixes the element $k+1$.

Now, by induction on the length of $cyc$, assume that our statement that
$\theta$ is $x--y$ equivalent to some $\phi\in S_{k+1}$, satisfying
$\phi(k+1)=k+1$, is true for all permutations $\theta$ with the length of
$cyc$ being $\leq m$.

Consider any $\theta$ with
\[
cyc=(k+1)\mapsto\theta_{j_{2}}\mapsto...\mapsto\theta_{j_{m+1}}\mapsto(k+1).
\]

If the length of the cycle $cyc$ is $\geq3$, then we select some
$b=\theta_{j_{r}}$ in $cyc$, satisfying $\theta_{j_{r+1}}\neq k+1$, such that
$b+1$ is not contained in $cyc$. If $cyc$ does not contain $0$, we take
$b=k+1$. If $cyc$ contains $0$, but does not contain $1$, we take $b=0$. And
so on. By our assumptions, $cyc$ cannot contain all the numbers between $0$
and $k$, otherwise $\theta^{\circ}$ will have only one cycle. Hence, we will
always find such $b$. Notice, that, since $cyc$ contains $k+1$, $b\neq k$.
Without loss of generality, assume that $b+1=\theta_{i_{1}}\in cyc_{2}$, where%
\[
cyc_{2}=(b+1)\mapsto\theta_{i_{2}}\mapsto...\mapsto\theta_{i_{a}}%
\mapsto(b+1).
\]

Let $s$ be such that $\theta_{s}=\theta_{j_{r+1}}+1\operatorname{mod}\left(
k+2\right)  $. Notice, that if $s>j_{r+1}$ then $s>j_{r+1}+1$. If
$\theta_{j_{r+1}}+1=\theta_{j_{r+1}+1}$ then in $\theta^{\circ}$ we would have
$\theta_{j_{r+1}}\mapsto\theta_{j_{r+1}}$, since $(0,1,\dots,n)$ takes
$\theta_{j_{r+1}}$ to $\theta_{j_{r+1}}+1\operatorname{mod}\left(  k+2\right)
$ and then $\theta^{\cdot}$ takes $\theta_{j_{r+1}+1}=\theta_{j_{r+1}%
}+1\operatorname{mod}\left(  k+2\right)  $ back to $\theta_{j_{r+1}}$. But
this contradicts our assumption on the length of $cyc$. We define permutation
$\tau\in S_{k+1}$ as follows:

\begin{itemize}
\item If $s>j_{r+1}$ then%
\[
\tau=\langle\theta_{1}\;\;...\;\;\theta_{j_{r+1}}\;\;\theta_{j_{r+1}%
+2}\;\;...\;\;\theta_{s-1}\;\;\theta_{j_{r+1}+1}\;\;(\theta_{j_{r+1}}%
+1=\theta_{s})\ \ ...\ \ \theta_{k+1}\rangle
\]

\item If $s<j_{r+1}$ then%
\[
\tau=\langle\theta_{1}\ \ ...\ \ \theta_{s-1}\ \ \theta_{j_{r+1}+1}%
\ \ (\theta_{j_{r+1}}+1=\theta_{s})\ \ ...\ \ \theta_{j_{r+1}}\ \ \theta
_{j_{r+1}+2}\ \ ...\ \ \theta_{k+1}\rangle
\]

\end{itemize}

Notice, that $\tau^{\circ}$ differs from $\theta^{\circ}$ only in the cycles
$cyc$ and in $cyc_{2}$. Namely, in $\tau^{\circ}$ these two cycles are
\[
cyc\prime=(k+1)\mapsto\theta_{j_{2}}\mapsto...\mapsto(\theta_{j_{r}}%
=b)\mapsto\theta_{j_{r+2}}\mapsto....\mapsto\theta_{j_{m+1}}\mapsto(k+1)
\]
and%
\[
cyc_{2}\prime=(b+1)\mapsto\theta_{i_{2}}\mapsto...\mapsto\theta_{i_{a}}%
\mapsto\theta_{j_{r+1}}\mapsto(b+1),
\]
respectively. Performing $\theta_{j_{r+1}}--(b+1)$ exchange operation on
$\tau$ produces $\theta$. But $\tau$ contains only $m$ elements in its cycle
$cyc\prime$, which contains the element $k+1$. Hence, by our induction
hypotheses on the length of $cyc$, $\tau$ is $x--y$ equivalent to some $\phi$,
which fixes the element $k+1$. Hence our $\theta$ is also $x--y$ equivalent to
this $\phi$. This completes the proof of the induction on the length of the
cycle $cyc$.

But we have already shown above, that any two permutations, which fix the
element $k+1$ and have the same number of alternating cycles in their cycle
graphs, are $x--y$ equivalent. Hence any two permutations, which have the same
number of alternating cycles in their cycle graphs, are $x--y$ equivalent. If
this number of alternating cycles is two or more, then these two permutations
are both $x--y$ equivalent to some permutation with the same number of
alternating cycles, which fixes $k+1$. And if this number of alternating
cycles in their cycle graphs is one, then we already established above, that
these two permutations are $x--y$ equivalent. Hence, our Theorem follows.
\end{proof}

At this point we are ready to state the main theorem of this work. This
theorem generalizes the observations, made in the end of the previous section,
that the probability of a permutation equality in a fixed finite group $G$
depends only on the number of the alternating cycles in the cycle graph the permutation.

\begin{theorem}
\label{main.theorem} (a) Let $\phi\in S_{n}$ be a permutation such that
$Gr(\phi)$ contains $k$ alternating cycles. Then
\[
Pr_{\phi}(G)=Pr^{n+1-k}(G)=Pr(a_{1}a_{2}\cdots a_{n-k}a_{n+1-k}=a_{n+1-k}%
a_{n-k}\cdots a_{2}a_{1})
\]
for $k\leq n$ (which implies that $k\leq n-1$, since, as it is shown in
\cite{DL}, $n-k$ is always an odd number), and $Pr_{\phi}(G)=1$ for $k=n+1$
(which implies that $\phi$ is the Identity permutation).

(b) Let $G$ be a non-abelian group, $\phi$ and $\theta$ two permutations in
$S_{n}$. Then, $Pr_{\phi}(G)=Pr_{\theta}(G)$ if and only if the number of
alternating cycles in the cycle graph $Gr\theta$ equals to that in the cycle
graph $Gr\phi$, which means the spectrum of the probabilities of permutation
equalities for permutations from $S_{n}$ in a general non-abelian group $G$
consists, in general, of exactly $[[\frac{n}{2}]]$ different values
corresponds to each different Hultman class in the cycle graph decomposition
of $S_{n}$.
\end{theorem}

\begin{proof}
\textit{(a)} Define the permutation
\[
\theta=\langle
n+1-k\;\;n-k\;\;...\;\;2\;\;1\;\;n+2-k\;\;n+3-k\;\;...\;\;n\rangle
\]
in $S_{n}$. It is easy to check, that $\theta^{\circ}$ has $k$ cycles. Hence,
$Gr(\theta)$ contains $k$ alternating cycles. By Theorem
\ref{hultman-exchange-cyclic}, this implies that $\phi$ and $\theta$ are
$x--y$ equivalent. From Theorems \ref{exchange} and \ref{cyclic}, we get that
$Pr_{\phi}(G)=Pr_{\theta}(G)$. Finally, observe that, since $\theta$ fixes
elements $a_{n+2-k},...,a_{n}$ in their places, all these elements cancel-out
in the permutation equality of $\theta$. Thus,
\[
Pr_{\theta}(G)=Pr(a_{1}a_{2}\cdots a_{n-k}a_{n+1-k}=a_{n+1-k}a_{n-k}\cdots
a_{2}a_{1}).
\]

\textit{(b)} Since, $Pr^{2n}(G)>Pr^{2n+2}(G)$ for every non-abeliann $G$ and
every $n\geq0$ (see \cite{DasNath2}), we get the result of the theorem.
\end{proof}

\section{Some explicit formulae for $Pr_{\langle2n\;\;...\;\;1\rangle}(G)$}

We end this work by providing two formulae for $Pr_{\langle
2n\;\;...\;\;1\rangle}(G)$. Our first formula expresses $Pr_{\langle
2n\;\;...\;\;1\rangle}(G)$ in terms of $Stab.Prod_{n}(x_{1},x_{2},\dots
x_{n})$. Our second formula expresses $Pr_{\langle2n\;\;...\;\;1\rangle}(G)$
in terms of $c_{i_{1},...,i_{n};j}(G)$. These results constitute a
generalization of what was shown in Theorem \ref{mainfour} for permutations
from $S_{4}$.

Let $G$ be a finite group.%
\begin{equation}
Pr(a_{2n}a_{2n-1}\cdots a_{1}=a_{1}a_{2}\cdots a_{2n})=\frac{\sum
\limits_{x_{1},x_{2},\dots,x_{n}\in G}|Stab.Prod_{n}(x_{1},x_{2},\dots
,x_{n}))|}{|G|^{2n}} \label{Eq.Stab.Prod}%
\end{equation}%
\begin{equation}
Pr(a_{2n}a_{2n-1}\cdots a_{1}=a_{1}a_{2}\cdots a_{2n})=\sum\limits_{i_{1}%
,i_{2},...,i_{n},j=1}^{c(G)}\frac{\left\vert \Omega_{j}\right\vert \cdot
c_{i_{1},i_{2},...,i_{n};j}^{2}(G)}{\left\vert \Omega_{i_{1}}\right\vert
\cdot\left\vert \Omega_{i_{2}}\right\vert \cdot\cdots\cdot\left\vert
\Omega_{i_{n}}\right\vert \cdot\left\vert G\right\vert ^{n}}
\label{Eq.weak.ring}%
\end{equation}

\begin{proof}
Let us consider a generic equation
\[
a_{1}a_{2}\cdots a_{2n}=a_{2n}a_{2n-1}\cdots a_{1}.
\]
For $i=1,...,n$, denote the product $a_{2i-1}a_{2i}$ by $x_{i}$ and the
product $a_{2i}a_{2i-1}$ by $x_{i}^{\prime}$. Notice, that by Lemma
\ref{commute-conjug}, we have $x_{i}\sim x_{i}^{\prime}$.

The equation
\[
a_{1}a_{2}\cdots a_{2n}=a_{2n}a_{2n-1}\cdots a_{1}%
\]
becomes
\[
x_{1}x_{2}\cdots x_{n}=x_{n}^{\prime}x_{n-1}^{\prime}\cdots x_{1}^{\prime}.
\]

Now consider any
\[
x_{1},...,x_{n},x_{1}^{\prime},...,x\prime_{n}\in G,
\]
such that $x_{i}\sim x_{i}^{\prime}$ for all $i$. Then, by Lemma
\ref{conjug-decomp}, there are $\frac{|G|}{|\Omega_{(}x_{i})|}=|C_{G}(x_{i})|$
different ways of breaking each $x_{i}$ into a product $a_{2i-1}a_{2i}$ in
such a way, that $x_{i}^{\prime}=a_{2i}a_{2i-1}$. Thus to each fixed equation
\[
x_{1}x_{2}\cdots x_{n}=x_{n}^{\prime}x_{n-1}^{\prime}\cdots x_{1}^{\prime}%
\]
correspond
\[
\left\vert C_{G}(x_{1})\right\vert \cdot\left\vert C_{G}(x_{2})\right\vert
\cdot\cdots\cdot\left\vert C_{G}(x_{n})\right\vert
\]
different equations
\[
a_{1}a_{2}\cdots a_{2n}=a_{2n}a_{2n-1}\cdots a_{1}.
\]

Notice, that for any fixed $x_{1},x_{2},...,x_{n}\in G$, we can take%
\[
x_{1}^{\prime\prime}=(x_{2}x_{3}\cdots x_{n})^{-1}x_{1}(x_{2}x_{3}\cdots
x_{n})=x_{n}^{-1}\cdots x_{3}^{-1}x_{2}^{-1}x_{1}x_{2}x_{3}\cdots x_{n}%
\]
and obtain an equation
\[
x_{1}x_{2}\cdots x_{n}=x_{n}x_{n-1}\cdots x_{2}x_{1}^{\prime\prime}.
\]

Now, for any equation
\[
x_{1}x_{2}\cdots x_{n}=x_{n}^{\prime}x_{n-1}^{\prime}\cdots x_{1}^{\prime}%
\]
as above, there exist some
\[
g_{1},g_{2},...,g_{n}\in G
\]
such that $x_{1}^{\prime}=g_{1}x_{1}^{\prime\prime}g_{1}^{-1}$ and
$x_{i}^{\prime}=g_{i}x_{i}g_{i}^{-1}$, for $i=2,3,...,n$.

Thus if we select and fix $x_{1},x_{2},...,x_{n}$, we will have%
\[
\frac{|Stab.Prod_{n}(x_{n},...,x_{2},x_{1}^{\prime\prime})|}{\left\vert
C_{G}(x_{n})\right\vert \cdot\cdots\cdot\left\vert C_{G}(x_{2})\right\vert
\cdot|C_{G}(x_{1}^{\prime\prime})|}%
\]

different equations
\[
x_{1}x_{2}\cdots x_{n}=x_{n}^{\prime}x_{n-1}^{\prime}\cdots x_{1}^{\prime},
\]
in which $x_{i}^{\prime}\sim x_{i}$ and for all $i$. Any ordered $n$-tuple%
\[
(g_{n},g_{n-1},...,g_{2},g_{1}^{\prime\prime})\in Stab.Prod_{n}(x_{n}%
,...,x_{2},x_{1}^{\prime\prime})
\]
produces an equation
\[
x_{1}x_{2}\cdots x_{n}=x_{n}^{\prime}x_{n-1}^{\prime}\cdots x_{1}^{\prime}%
\]
by $x_{i}^{\prime}=g_{i}x_{i}g_{i}^{-1}$, for $i=n,...,2$, and $x\prime
_{1}=g_{1}x_{1}^{\prime\prime}g_{1}^{-1}$. And any two such equations are
equal if and only if $g_{i}\in C_{G}(x_{i})$ and $g_{1}\in C_{G}%
(x^{\prime\prime})$.

Notice, that $\left\vert C_{G}(x_{1}^{\prime\prime})\right\vert =|C_{G}%
(x_{1})|$. Thus, for each fixed ordered $n$-tuple $(x_{1},x_{2},...,x_{n})$ of
elements of $G$ we have%
\[
\frac{|Stab.Prod_{n}(x_{n},...,x_{2},x_{1}^{\prime\prime})|}{\left\vert
C_{G}(x_{n})\right\vert \cdot\cdots\cdot\left\vert C_{G}(x_{2})\right\vert
\cdot\left\vert C_{G}(x_{1})\right\vert }%
\]

different equations
\[
x_{1}x_{2}\cdots x_{n}=x_{n}^{\prime}x_{n-1}^{\prime}\cdots x_{1}^{\prime}.
\]
Moreover, to each one of these equations correspond
\[
\left\vert C_{G}(x_{1})\right\vert \cdot\left\vert C_{G}(x_{2})\right\vert
\cdot\cdots\cdot\left\vert C_{G}(x_{n})\right\vert
\]
different equations
\[
a_{1}a_{2}\cdots a_{2n}=a_{2n}a_{2n-1}\cdots a_{1},
\]
as shown above. Thus to each ordered $n$-tuple $(x_{1},x_{2},...,x_{n})$ of
elements of $G$ correspond\newline%
\[
|Stab.Prod_{n}(x_{n},...,x_{2},x_{1}^{\prime\prime})|
\]
different equations
\[
a_{1}a_{2}\cdots a_{2n}=a_{2n}a_{2n-1}\cdots a_{1}.
\]

Hence, to find
\[
Pr(a_{1}a_{2}\cdots a_{2n}=a_{2n}a_{2n-1}\cdots a_{1})
\]
we need to sum the value\newline%
\[
|Stab.Prod_{n}(x_{n},x_{n-1},\dots x_{2},x_{1}^{\prime\prime})|
\]
over all $x_{1},x_{2},...,x_{n}\in G$.

Since $x_{1},x_{2},...,x_{n}$ run through all the elements of $G$, and for
each fixed $x_{1},...x_{n-1}$ there is one-to-one correspondence between
$x_{1}^{\prime\prime}$ and $x_{1}$, this is the same as to sum the value
\[
|Stab.Prod_{n}(y_{1},y_{2},\dots,y_{n})|
\]
as over all $y_{1},y_{2},...,y_{n}$ run through all the elements of $G$.

Actually, we just renamed $x_{n-i}$ by $y_{i}$. Therefore,%
\[
Pr(a_{1}a_{2}\cdots a_{2n}=a_{2n}a_{2n-1}\cdots a_{1})=\frac{\sum
\limits_{y_{1},y_{2},\dots,y_{n}\in G}|Stab.Prod_{n}(y_{1},y_{2},\dots
,y_{n})|}{|G|^{2n}}.
\]

Now, select and fix an element $z$ in some equivalence class $\Omega_{j}$ of
$G$. For each $i$, there are $|C_{G}(x_{i})|$ different ways to break $x_{i}$
into a product $a_{2i-1}a_{2i}$ in such a way that $a_{2i}a_{2i-1}%
=x_{i}^{\prime}$. Hence, for each (fixed) equation
\[
x_{1}x_{2}\cdots x_{n}=z=x_{1}^{\prime}x_{2}^{\prime}\cdots x_{n}^{\prime}%
\]
we have exactly
\[
\left\vert C_{G}(x_{1})\right\vert \cdot\left\vert C_{G}(x_{2})\right\vert
\cdot\cdots\cdot\left\vert C_{G}(x_{n})\right\vert
\]
different equations%
\begin{gather*}
(a_{1}a_{2})(a_{3}a_{4})\cdots(a_{2n-1}a_{2n})=x_{1}x_{2}\cdots x_{n}=\\
z=x_{1}^{\prime}x_{2}^{\prime}\cdots x_{n}^{\prime}=(a_{2}a_{1})(a_{4}%
a_{3})\cdots(a_{2n}a_{2n-1}).
\end{gather*}

In other words,%
\[
\left\vert C_{G}(x_{1})\right\vert \cdot\left\vert C_{G}(x_{2})\right\vert
\cdot\cdots\cdot\left\vert C_{G}(x_{n})\right\vert =\frac{|G|^{n}}{\left\vert
\Omega(x_{1})\right\vert \cdot\left\vert \Omega(x_{2})\right\vert \cdot
\cdots\cdot\left\vert \Omega(x_{n})\right\vert }.
\]

Thus we obtain%
\begin{gather*}
Pr(a_{1}a_{2}a_{3}a_{4}\cdots a_{2n-1}a_{2n}=a_{2}a_{1}a_{4}a_{3}\cdots
a_{2n}a_{2n-1})=\\
\sum\limits_{i_{1},i_{2},...,i_{n},j=1}^{c(G)}\frac{\left\vert \Omega
_{j}\right\vert \cdot c_{i_{1},i_{2},...,i_{n};j}^{2}(G)}{\left\vert
\Omega_{i_{1}}\right\vert \cdot\left\vert \Omega_{i_{2}}\right\vert
\cdot\cdots\cdot\left\vert \Omega_{i_{n}}\right\vert \cdot\left\vert
G\right\vert ^{n}}.
\end{gather*}

There are $c_{i_{1},i_{2},...,i_{n};j}(G)$ different ways to break $z$ as
$x_{1}x_{2}\cdots x_{n}$ and $c_{i_{1},i_{2},...,i_{n};j}(G)$ different ways
to break $z$ as $x_{1}^{\prime}x_{2}^{\prime}\cdots x_{n}^{\prime}$. Thus, for
each $n$-tuple $\Omega_{i_{1}},\Omega_{i_{2}},...,\Omega_{i_{n}}$ of conjugacy
classes of $G$, there are $c_{i_{1},i_{2},...,i_{n};j}^{2}(G)$ different
equations
\[
x_{1}x_{2}\cdots x_{n}=z=x_{1}^{\prime}x_{2}^{\prime}\cdots x_{n}^{\prime}%
\]
with $x_{t},x_{t}^{\prime}\in\Omega_{i_{j}}$ for all $t$. Since both
\[
\phi=\langle2n\;\;2n-1\;\;...\;\;1\rangle
\]
and
\[
\theta=\langle2\;\;1\;\;4\;\;3\;\;...\;\;2n\;\;2n-1\rangle
\]
have exactly one alternating cycle in their cycles graphs $Gr(\phi)$ and
$Gr(\theta)$, respectively, from Theorem \ref{main.theorem} we deduce%
\begin{gather*}
Pr(a_{1}a_{2}\cdots a_{2n}=a_{2n}a_{2n-1}\cdots a_{1})=\\
Pr(a_{1}a_{2}a_{3}a_{4}\cdots a_{2n-1}a_{2n}=a_{2}a_{1}a_{4}a_{3}\cdots
a_{2n}a_{2n-1})=\\
\sum\limits_{i_{1},i_{2},...,i_{n},j=1}^{c(G)}\frac{\left\vert \Omega
_{j}\right\vert \cdot c_{i_{1},i_{2},...,i_{n};j}^{2}(G)}{\left\vert
\Omega_{i_{1}}\right\vert \cdot\left\vert \Omega_{i_{2}}\right\vert
\cdot\cdots\cdot\left\vert \Omega_{i_{n}}\right\vert \cdot\left\vert
G\right\vert ^{n}}.
\end{gather*}

\end{proof}

\section{Conclusions and Future Work}

Now, after giving an explicit formula for $Pr^{2n}(G)$, we recall some known
results connecting $Pr^{2n}(G)$ to the commutator subgroup $G^{\prime}$, the
quotient $G/Z(G)$, and isoclinism of groups.

\begin{theorem}
\cite{DasNath2}\label{Commutator} $lim_{n\rightarrow\infty}Pr^{2n}(G)=\frac
{1}{|G^{\prime}|}$.
\end{theorem}

\begin{theorem}
\cite{DasNath2}\label{Thm_Isoclinic} Let $G_{1}$ and $G_{2}$ two finite groups
such that $G_{1}$ and $G_{2}$ are isoclinic, then $Pr^{2n}(G_{1}%
)=Pr^{2n}(G_{2})$ for every $n\geq0$.
\end{theorem}

For example, every two abelian groups are isoclinic, and $Pr^{2n}(G)=1$ for
every abelian group. Moreover, the two non-abelian groups of order $8$, namely
the Dihedral group $D_{4}$ and the Quaternion group $Q_{8}$ are isoclinic, and
we mentioned $Pr^{2n}(D_{4})=Pr^{2n}(Q_{8})$ for every $n\geq0$.

Now, notice that the opposite direction of the second result is not true.
There exists two groups $G_{1}$ and $G_{2}$ where both of them have order $64$
such that $Pr^{2n}(G_{1})=Pr^{2n}(G_{2})$ for every $n\geq0$, but $G_{1}$ and
$G_{2}$ are not isoclinic.%

\begin{gather*}
G_{1}=\langle a_{1},a_{2},a_{3},a_{4},a_{5},a_{6}:a_{1}^{2}=a_{4},a_{2}%
^{2}=a_{3}^{2}=a_{5}^{2}=a_{6}^{2}=1,a_{4}^{2}=a_{6},\\
\lbrack a_{1},a_{2}]=a_{3},[a_{1},a_{3}]=[a_{2},a_{4}]=a_{5},[a_{1}%
,a_{4}]=[a_{1},a_{5}]=[a_{1},a_{6}]=\\
\lbrack a_{2},a_{3}]=[a_{2},a_{5}]=[a_{2},a_{6}]=[a_{3},a_{4}]=[a_{3}%
,a_{5}]=[a_{3},a_{6}]=\\
\lbrack a_{4},a_{5}]=[a_{4},a_{6}]=[a_{5},a_{6}]=1\rangle.
\end{gather*}

\begin{gather*}
G_{2}=\langle a_{1},a_{2},a_{3},a_{4},a_{5},a_{6}:a_{1}^{2}=a_{2}^{2}%
=a_{3}^{2}=a_{4}^{2}=a_{5}^{2}=a_{6}^{2}=1,\\
\lbrack a_{1},a_{2}]=a_{6},[a_{1},a_{3}]=[a_{2},a_{4}]=a_{5},[a_{1}%
,a_{4}]=[a_{1},a_{5}]=[a_{1},a_{6}]=\\
\lbrack a_{2},a_{3}]=[a_{2},a_{5}]=[a_{2},a_{6}]=[a_{3},a_{4}]=[a_{3}%
,a_{5}]=[a_{3},a_{6}]=\\
\lbrack a_{4},a_{5}]=[a_{4},a_{6}]=[a_{5},a_{6}]=1\rangle.
\end{gather*}

$G_{1}$ and $G_{2}$ are not isoclinic (even not weakly!) since $G_{1}%
/Z(G_{1})$ is non-abelian group of order $16$, and $G_{2}/Z(G_{2})$ is the
elementary abelian group of order $16$. Therefore, $G_{1}/Z(G_{1})$ can not be
isomorphic to $G_{2}/Z(G_{2})$, although $\left\vert G_{1}/Z(G_{1})\right\vert
=\left\vert G_{2}/Z(G_{2})\right\vert $.

By Theorem \ref{Commutator} we know that if $Pr^{2n}(G_{1})=Pr^{2n}(G_{2})$
for every $n\geq0$, then $\left\vert G_{1}^{\prime}\right\vert =\left\vert
G_{2}^{\prime}\right\vert $. It motivates the following.

\begin{conjecture}
\label{isoclinic} Let $G_{1}$ and $G_{2}$ are two finite groups such that
$Pr^{2n}(G_{1})=Pr^{2n}(G_{2})$ for every $n\geq0$, then $\left\vert
G_{1}/Z(G_{1})\right\vert =\left\vert G_{2}/Z(G_{2})\right\vert $.
\end{conjecture}

The groups $G_{2}$ and $G_{3}=D_{8}\times D_{8}$ give a counterexample to the
opposite direction of Conjecture \ref{isoclinic}.

\begin{claim}
\bigskip$G_{2}$ and $G_{3}=D_{8}\times D_{8}$ are two groups of order $64$
that are weakly isoclinic but not isoclinic.
\end{claim}

\begin{proof}
Direct calculations show that $G_{2}/Z(G_{2})$ is isomorphic to $G_{3}%
/Z(G_{3})$, $G_{2}^{\prime}$ is isomorphic to $G_{3}^{\prime}$.

On the other hand, $Pr^{2}(G_{2})=\frac{22}{64}$ and $Pr^{2}(G_{3})=\frac
{25}{64}$. Hence, by Theorem \ref{Thm_Isoclinic} $G_{2}$ is not isoclinic to
$G_{3}$.
\end{proof}

\end{document}